\documentclass[12pt]{article}

\usepackage{amsmath}
\usepackage{amsfonts}
\usepackage{amssymb}
\usepackage{amsthm}
\usepackage{graphics}
\usepackage{subfig}
\usepackage{graphicx}
\usepackage{float}
\usepackage{xcolor}
\usepackage{amsthm}
\usepackage{mathrsfs}
\usepackage{subfig}
\usepackage{framed}
\usepackage{multirow}

\def\prob{\mathbb P} 
\def\N{\mathbb N}

\def\R{\mathbb R}

\def\Z{\mathbb Z}

\def\x{\mathbf x}
\def\y{\mathbf y}
\def\u{\mathfrak u}
\def\dx{\triangle x}
\def\dy{\triangle y}
\def\dt{\triangle t}
\def\hU{\widehat{U}}
\def\bU{\overline{U}}

\newcommand{\pder}[2][]{\frac{\partial#1}{\partial#2}}
\newcommand{\pderr}[2][]{\frac{\partial^2 #1}{\partial#2 ^2}}
\newcommand{\norm}[2][]{||#2||_#1}
\newcommand{\cso}[1]{{\color{black}#1\normalcolor}}
\newcommand{\cmy}[1]{{\color{black}#1\normalcolor}}

\newtheorem{theorem}{Theorem}
\newtheorem{lemma}{Lemma}

\newcommand*{\email}[1]{\texttt{#1}}
\title{\vspace{-1in}Computational Framework to Capture the Spatiotemporal Density of Cells with a Cumulative Environmental Coupling
}

\date{\today}
\author{Michael A. Yereniuk \thanks{Department of Mathematical Sciences, Worcester Polytechnic Institute, Worcester, MA 
  (\email{mayereniuk@wpi.edu})}
\and Sarah D. Olson \thanks{Department of Mathematical Sciences, Worcester Polytechnic Institute, Worcester, MA 
  (\email{sdolson@wpi.edu})}}

\begin{document}
\maketitle

\begin{abstract}
Stochastic agent-based models can account for millions of cells with spatiotemporal movement that can be a function of different factors. However, these simulations can be computationally expensive. In this work, we develop a novel computational framework to describe and simulate stochastic cellular processes that are coupled to the environment. Specifically, through upscaling, we derive a continuum governing equation that considers the cell density as a function of time, space, and a cumulative variable that is coupled to the environmental conditions. For this new governing equation, we consider the stability through an energy analysis, as well as proving uniqueness and well-posedness. To solve the governing equations in free-space, we propose a numerical method using fundamental solutions. As an application, we study a cell moving in an infinite domain that contains a toxic chemical, where a cumulative exposure above a critical value results in cell death. We illustrate the validity of this new modeling framework and associated numerical methods by comparing the density of live cells to results from the corresponding agent-based model. 
\end{abstract}

\begin{center}
\end{center}

\section{Introduction}\label{sec:Intro}
In many applications, we may wish to track the individual movement of a collection of cells or agents that exhibit state changes and interact with a dynamic environment. A Lagrangian framework such as a cellular automaton or an agent-based  model can be used to track the movement of each cell, while also 
accounting for state changes \cite{Hinkelmann2011,Holcombe,Laubenbacher2012, Packard1985}. To implement this type of model, one would define the cell as the agent and a new location at each time point is chosen based on pre-defined rules. With current computer architectures, simulating millions of cells is feasible but is offset by the computational time required to compute a sufficient number of simulations for analysis \cite{Alden13,Cosgrove15,Read12,Vargha00}.  In addition, in the case where one wants to couple movement rules to evolving profiles on the domain, as well as tracking cumulative variables related to environmental conditions in multiple dimensions, computational time could be prohibitive. 

Several other modeling approaches can also be utilized to understand the dynamics of a large group of cells or agents that exhibit state changes and interact with their environment; each of these methods has their own limitations and advantages. The  transition between cell states can be modeled as a stochastic process, and in the case of a Markovian or memoryless process, the transition rate  will only depend on the current state, following a Poisson distributed random process \cite{IyerBiswas16,Polizzi16}. In terms of a discrete-time Markov Chain, each state can correspond to a particular combination of cell state and cell location. Using this type of approach,  the probability of a cell being in a given state at a given time can be determined and the master equation for the rate of change of the associated probability can be obtained. 
Since analysis of cellular processes is \cmy{often times} easier in the continuous setting, one can easily move from discrete probabilities of cell states  to a continuous probability density of cell states by assuming a continuum of cell states. However, in order to keep track of a cumulative variable accounting for interactions with the environment, this approach would break down if memory in the system was necessary. 
 
Random walk models are also stochastic processes that consist of sequential random steps of movement; they have been widely used to investigate cellular motility, often in a spatially homogeneous environment \cite{Berg83,Codling08,Montroll84,Vuilleumier06,Weiss83}. 
Assuming the moving agent or cell is memoryless, an equation governing the spatiotemporal evolution of the density of cells can be determined. It corresponds to a standard diffusion equation if there is no bias in the motion or an advection-diffusion equation if there is bias in the motion \cite{Berg83,Codling08,Montroll84,Weiss83}. The advection-diffusion equation can capture different \textit{taxis}, biasing the probability of movement based on chemical profiles (chemotaxis), temperature gradients (thermotaxis), fluid flow (rheotaxis), or environmental mechanical stiffness (durotaxis) \cite{Alt80,Cai06,Okubo01,Othmer88,Othmer97}.  The continuum limit of the stochastic process is 
often formulated in the case of cell motility since it is tractable from an analytical perspective and we have existing computational methods to easily solve these governing equations. In this framework, accounting for different cell states would correspond to a system of coupled partial differential equations where local sinks or sources would describe leaving one cell state and entering another cell state. Currently, there is no random walk modeling framework to account for cells changing states due to a cumulative environmental coupling. 

Cellular processes, such as absorbing chemicals or nutrients, moving, and state transitioning will depend on the local and dynamic environment \cite{Morales07,Paul17,Selmeczi08,Stebbings01}. Often, exposure to a chemical or drug beyond a threshold will cause a change in state or motion, thus it is important to capture the cumulative chemical exposure.
To date, analysis has primarily focused on the motion of a single cell type in a homogeneous environment or a sequence of cell states in time (not accounting for motion) \cite{Codling08,Polizzi16}. 
On the other hand, stochastic agent-based models can simulate many cells and states, but analyzing these models is intractable and computations  can become prohibitive with a large number of cells coupled to the evolving environment. \cso{Hence, one approach to resolve this issue is by implementing multiscale models that bridge the gap between the individual cell level and dynamics at the global continuum level. For example, recent work has focused on developing movement rules coupled to chemical concentrations in the case of bacterial chemotaxis \cite{Xue16,Xue18a,Xue18b}}.

In this paper, we focus on the development of new computational methods to capture the continuum  density of cells or agents in space, accounting for the cumulative exposure to the environment as a continuous variable. As outlined in Section \ref{sec:prelim}, our motivating example will be a cell that moves and absorbs chemicals; a state change occurs when the cell has absorbed a critical (toxic) threshold of chemicals, causing the cell to die. In Section \ref{sec:model}, we will show how the new governing equation is able to capture these dynamics. An analysis of this equation is detailed in Section \ref{sec:analysis} and the numerical method is outlined in Section \ref{sec:numerics}. Representative numerical results are given in Section \ref{sec:results}, comparing computation of our new governing equations with the corresponding agent-based model for the case of cells that randomly move and absorb chemical from the surrounding environment. \cso{Additional commentary is given for several limiting cases in Section \ref{sec: scaling}.}
\\ \\
\textbf{Notation}. $\N$ denotes the natural numbers and $\R^n$ denotes $n$-dimensional Euclidean space. All vectors will be denoted with bold face, i.e.  $\mathbf{x}=[x_1,x_2,\ldots,x_n]^T$ when $\x\in\R^n$ and the superscript $T$ denotes the vector transpose. The $L^p$ space is defined by generalizing the $p$-norm for vector spaces $\R^n$, whereas $C^k$ defines the space of continuously differentiable functions up to the $k^{th}$ derivative. For ease of notation, we define $\Omega^n = \R^n \times [0,\infty)$, the spatial and absorption domains. The convolution of two functions $f$ and $g$ is denoted as $f*g$ and we define $*^m$ as
\[
f*^m g
\equiv
\underbrace{f*f* \cdots *f}_{m \text{ times}} * g.
\]
We will use $\delta(\cdot)$ to denote the Dirac delta distribution. To simplify, we will denote 
$\norm[1]{\cdot} \equiv \int_{\Omega^n} | \cdot | \,d\xi\,d\x$ and $\norm[2]{\cdot} \equiv \int_{\Omega^n}(\cdot)^2 \,d\xi\,d\x$. The error function is defined as $erf(z)\equiv(2/\sqrt{\pi})\int_0^ze^{-t^2}dt$. The evaluation of $\lfloor d \rfloor$ gives the greatest integer that is less than or equal to $d$.

\section{Problem formulation and preliminaries}\label{sec:prelim}
Suppose a spatial domain contains a spatially-varying (or time-varying) chemical concentration.  For simplicity, we will assume the chemical concentration is a positive, spatially-dependent distribution, $C(\mathbf{x})$.
We then insert a cell in this domain at location $\x_0$ - to avoid semantic confusion later in this paper, we will refer to this cell as an agent.  A schematic of this setup is shown in Fig.~\ref{fig:schematic}.
This agent has a given probability of moving at each time point. In the 2-dimensional setup, the agent may remain stationary or move left, right, up, or down as shown with the dotted arrows in Fig.~\ref{fig:schematic}. 
After moving, the agent will absorb a certain amount of chemical according to a function $\hat{\beta}(x)$, which depends on the local chemical concentration.  
Although we are not fixing a specific form of the function $\hat{\beta}(x)$ for our analysis, we will assert that this function preserves the property that if $C(x)>0$, then $\hat{\beta}(x)>0$.
When the cumulative absorption within the agent reaches a critical threshold, the agent changes state (e.g. the agent dies).

\begin{figure}[H]
     \centering
     \includegraphics[width=0.4\textwidth]{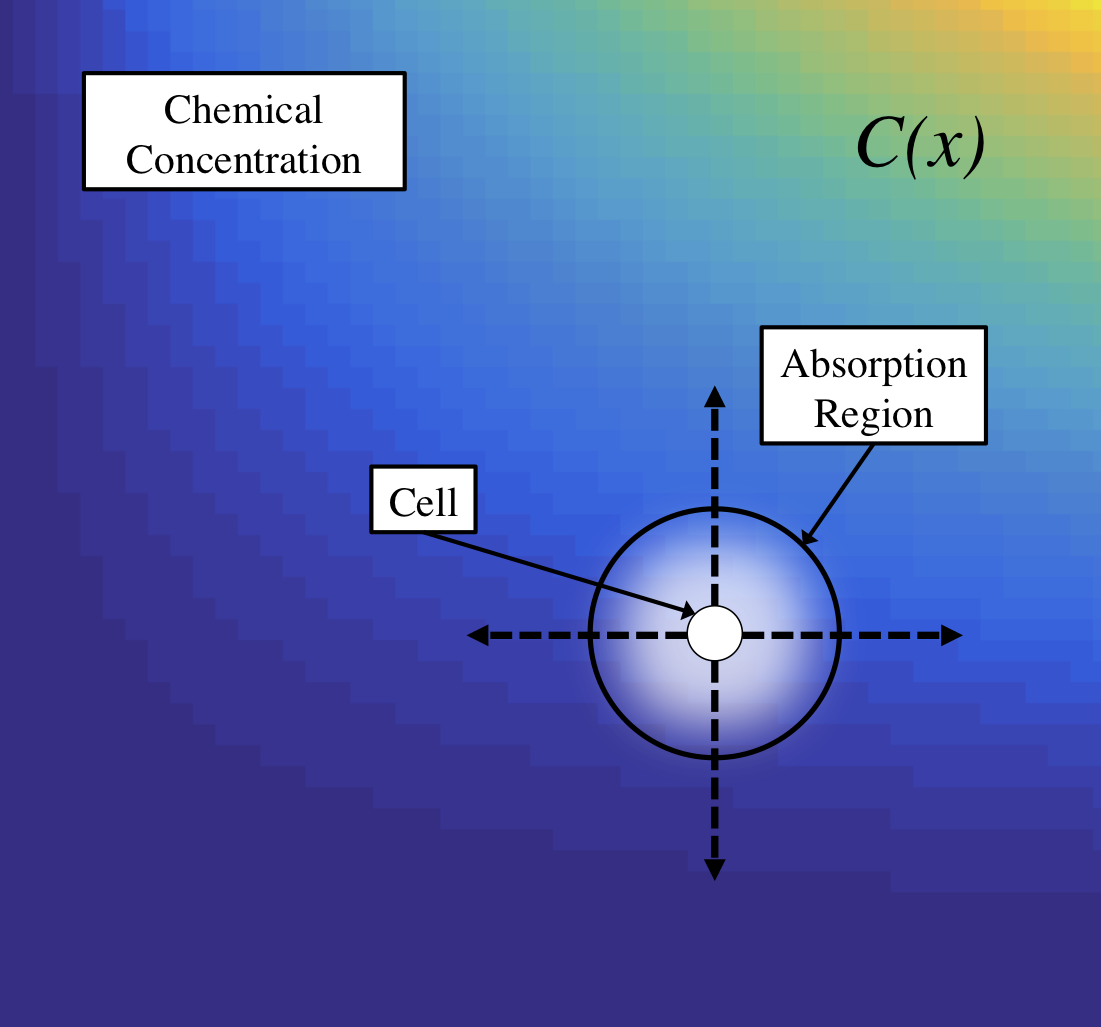}
     \caption{Schematic of an agent \cso{(the cell)} in a 2-dimensional domain that has a spatially dependent chemical profile. At each time step, the agent will move and then absorb chemical from the local, surrounding absorption region.}\label{fig:schematic}
\end{figure}
At first glance, we may consider the agent in our absorption model as having only two states: live and dead.  However, it is subtly more complicated.
Considering that the chemical concentration could be spatially and/or time-dependent, the particular path along which an agent travels may affect the amount of chemical the agent absorbs.  
For example, suppose there is no chemical concentration to the left of $\x_0$ and there is chemical to the right of $\x_0$.  Further, consider two distinct paths the agent may travel: in one path, the agent is contained within the left side of the domain and terminates at $\x_0$ at time $t$, whereas in another path, the agent is contained within the right side of the domain and terminates at $\x_0$ at $t$.  The agent which traveled along the first path will not have absorbed any chemical by time $t$, but the agent which traveled along the second path will absorb some chemical particles.

We need to account for varying amounts of chemical concentration by treating the amount absorbed as distinct states.  
However, as opposed to a compartmental model, the chemical concentration is a continuous variable.  
In order to account for this, we need to consider cumulative absorption as a dimension orthogonal to both the temporal and spatial dimensions.  
That is, the agent is at a location $\x$, having a cumulative chemical absorption $\xi$, at a particular time $t$.

The cumulative amount absorbed is path dependent. 
However, we cannot say that $\xi$ is dependent on space or time, just as we cannot say that $\x$ is dependent on time.  
All three variables are linked by our model, but should be considered independent.


\section{The continuum model}\label{sec:model}
Through upscaling, we will derive the continuum absorption model. We develop a discrete difference equation from the agent-based model (ABM), and then take the continuum limit, an approach used for standard random walk models \cite{Codling08}.
At each iteration, our ABM agent $\u$ moves in the domain with spatial-step $\dx$ and time-step $\dt$.  
Then, $\u$ absorbs chemical particles at its new location $\x$ based on a spatially-dependent function $\hat{\beta}(\x)$ that depends on the chemical distribution, $C(\x)$.  \cmy{We interpret $\hat{\beta}(\x)$ as the amount the cell absorbs during the time step of length $\dt$.}

\subsection{Derivation \cmy{of single-state absorption model}}\label{sec: Derivation}

Assume the ABM agent $\u$ is initialized in the 1-D spatial domain $\R$.  
We define the variable $U(x,t,\xi)$ to denote the density of agent $\u$ at location $x$ at time $t$, having absorbed $\xi$ total particles.
Suppose $\u$ moves a distance $\dx$ to the right to be at $x$ at the next time increment, $t+\dt$.  
We then have that $\u$ absorbs $\hat{\beta}(x)$ chemical particles at $x$, with a cumulative absorption of $\xi + \hat{\beta}(x)$ particles at time $t+\dt$ (when $\u$ had a cumulative absorption of $\xi$ at time $t$).  Similarly, if $\u$ moves a distance $\dx$ to the left to be at $x$ at the next time increment, then $\u$ will have cumulatively absorbed $\xi + \hat{\beta}(x)$ particles.
Letting $\ell(x)$ be the probability of moving left at location $x$ and $r(x)$  the probability of moving right at $x$, we can assert a difference equation modeling this behavior:
\begin{equation}
\begin{split}
U(x,t+&\dt,\xi+\hat{\beta}(x)) =
\ell(x+\dx)
U(x+\dx,t,\xi)  \\
& + r(x-\dx)U(x-\dx,t,\xi) +
[1-r(x)-\ell(x)]U(x,t,\xi),
\end{split}
\label{eq: discrete equation}
\end{equation}
\noindent where $\ell(x)+r(x)\leq 1$ for all $x\in\R$.
The $U(x,t+\dt,\xi+\hat{\beta}(x))$ term expresses the fact that agent $\u$ absorbs an additional $\hat{\beta}(x)$ chemical particles at location $x$ and time $t+\dt$.  
The right-hand side of \eqref{eq: discrete equation} accounts for all the different possible ways (along with their respective probabilities) that agent $\u$, having absorbed $\xi$ total number of chemical particles at time $t$ can be at location $x$ at time $t+\dt$. 

Assuming $U\in C^2(\R, [0,\infty), [0,\infty))$ and $q(x)\in C^2([0,1])$ for $q(x)$ as either $\ell(x)$ or $r(x)$, we can perform a Taylor expansion on $U$ and the moving probability $q(x)$ in  \eqref{eq: discrete equation} and get
\begin{align*}
U(x,t+\dt,\xi+\hat{\beta}(x)) &=
U(x,t,\xi) + \dt \pder[]{t}U(x,t,\xi) + \hat{\beta}(x)\pder[]{\xi}U(x,t,\xi) \\
&\qquad + \mathcal{O}(\dt^2, \hat{\beta}(x)^2),
\\
U(x\pm \dx,t,\xi) &= U(x,t,\xi) \pm \dx \pder[]{x}U(x,t,\xi) + \frac{\dx^2}{2}\pderr[]{x}U(x,t,\xi) \\
&\qquad + \mathcal{O}(\dx^3),\\
q(x+\dx)&=q(x)+\dx\pder[]{x}q(x)+\frac{\dx^2}{2}\pderr[]{x}q(x)+ \mathcal{O}(\dx^3).
\end{align*}
Inserting these expansions into \eqref{eq: discrete equation} results in
\begin{align*}
U + \dt \pder[U]{t} + \hat{\beta}(x)&\cmy{\pder[U]{\xi}} + \mathcal{O}\left(\dt^2, \hat{\beta}(x)^2\right)
= [\ell U] + \dx \pder[(\ell U)]{x} + 
\frac{\dx^2}{2}\pderr[(\ell U)]{x} 
\\
& + [r U] - \dx \pder[(r U)]{x} + \frac{\dx^2}{2}\pderr[(r U)]{x} \\
& +
[1-\ell-r]U +
\mathcal{O}\left(\dx^3 \right).
\end{align*}
Rearranging terms and simplifying gives us
\begin{align*}
\dt \pder[U]{t} + \hat{\beta}(x)&\pder[U]{\xi} = 
\dx \pder[(\ell-r)U]{x}  
\\
& + \frac{\dx^2}{2} \pderr[(\ell+r)U]{x} +
\mathcal{O}\left(\dx^3, \dt^2, \hat{\beta}(x)^2\right).
\end{align*}
\noindent If we define $\beta(x) = \hat{\beta}(x)/\dt$ and rearrange terms, we have
\begin{small}
\begin{equation}
\cmy{\pder[U]{t} + \beta(x)\pder[U]{\xi}}  = 
\frac{\dx}{\dt}\pder[(\ell-r)U]{x} +
\frac{\dx^2}{2\dt}\pderr[(\ell+r)U]{x}+
 \mathcal{O}\left(\dx^3, \dt,  \hat{\beta}(x)^2\right).
 \label{eq: biasedPDE}
\end{equation}
 \end{small}
 
For simplification, we will let $\ell(x)=r(x)=1/2$ for all $x\in \R$.  Our equation then reduces to:
\begin{equation}
U_t + \beta(x)U_\xi  = 
\frac{\dx^2}{2\dt}U_{xx} +
 \mathcal{O}\left(\dx^3, \dt,  \hat{\beta}(x)^2\right).
\end{equation}
We assume that $\dt \sim \dx^2$ and $ \hat{\beta}(x)= \mathcal{O}(\dt)$.  
Taking the limit as $\dt, \dx \to 0$ results in the following governing continuum equation:
\begin{equation}
U_t + \beta(x)U_\xi = DU_{xx},
\end{equation}
where $D = \lim_{\dt,\dx\to 0}\dx^2/(2\dt)$.
For this paper, we assume far-field boundary conditions and the initial condition can depend on $x$ and $\xi$. 

Hence, our partial differential equation (PDE)
for chemical absorption is as follows
\begin{equation}
\begin{cases}
& U_t + \beta(x)U_\xi = DU_{xx}, \hspace{1.5cm} (x,\xi)\in \Omega^1, t>0\\
& U = \phi(x,\xi), \hspace{3.15cm} (x,\xi)\in \Omega^1, t=0 \\
& \lim_{|x|\to \infty}U = 0, \hspace{1.7cm}\qquad (x,\xi)\in \Omega^1, t>0.
\end{cases}
\label{eq: PDE1d}
\end{equation}
\noindent In a similar way, assuming that the spatial step $\dx$ is the same in every direction, we can derive a continuum PDE in $n$ spatial dimensions.  The resulting PDE is as follows:
\begin{equation}
\begin{cases}
& U_t + \beta(\x)U_\xi = D_n \nabla^2 U, \qquad (\x,\xi)\in \Omega^n, t>0 \\
& U = \phi(\x,\xi), \hspace{2.85cm} (\x,\xi)\in \Omega^n, t=0 \\
& \lim_{|\x|\to \infty}U = 0, \hspace{2.25cm} (\x,\xi)\in \Omega^n, t>0,
\end{cases}
\label{eq: PDE}
\end{equation}
where $D_n = \lim_{\dx,\dt\to 0}\dx^2 / (2n\dt)$. 
\cmy{Here we should note that the above equations \eqref{eq: PDE1d} and \eqref{eq: PDE} only model cumulative absorption, without taking into account any state transitions due to this cumulative absorption.  We address two possible methods to account for state changes in the following subsection.}

\subsection{Absorption threshold}\label{sec: Absorption}
Now that we have a governing equation for tracking the cumulative absorption property of an agent, we can address possible absorption-dependent state-changes.
Suppose the agent changes state if the cumulative chemical absorption is greater than some absorption capacitance, $\xi_c$. That is, a cell is initially in the live state if $\xi<\xi_c$ and switches to a different state, possibly dying, if $\xi\geq \xi_c$.  We will denote the spatially-varying, total density of agents in the live state as $p(\x,t)$.  

\cmy{We have two possible methods\footnote{\cmy{We consider methods that keep $\beta$ from being dependent on the cumulative absorption amount, $\xi$.  One could derive a formula where $\beta$ depends on both the chemical concentration, $C(x)$, and $\xi$ to ensure that cells do not absorb chemical beyond the threshold.  However, the derivation and resulting Taylor series would produce additional terms that could make the analysis and numerical solutions more difficult.}
} 
for calculating the density of cells in a live state: by modifying our PDE model or by solving the PDE model and integrating $\xi$ over the domain of interest, $[0, \xi_c]$.} \cso{
First, we will examine modifying our PDE model and in this case, we will use the variable $V(x,t,\xi)$ to denote the density of agents in the live state (rather than the variable $U$ to avoid confusion in later sections).} \cmy{ Our difference equation includes a Heaviside function centered at $\xi_c$, since the agent switches state if $\xi \geq \xi_c$,
\begin{equation}
\begin{split}
V(x,t+&\dt,\xi+\hat{\beta}(x)) =
\ell(x+\dx)
V(x+\dx,t,\xi)  \\
& + r(x-\dx)V(x-\dx,t,\xi) +
[1-r(x)-\ell(x)]V(x,t,\xi) \\
& - H(\xi-\xi_c)V(x,t,\xi),
\end{split}
\label{eq: discrete equation V}
\end{equation}
where $H(\xi-\xi_c) = \begin{cases}1 &: \text{ if }\xi \geq \xi_c \\
0 &: \text{ otherwise} \end{cases}$ denotes the Heaviside function centered at $\xi_c$.  With the same assumptions as in the previous section and expanding in a Taylor series, we obtain the following PDE modeling live agents following an unbiased random walk:
\begin{equation}
\begin{cases}
& V_t + \beta(x)V_\xi = DV_{xx} - \tilde{H}(\xi-\xi_c)V, \hspace{1.5cm} (x,\xi)\in \Omega^1, t>0\\
& V = \phi(x,\xi), \hspace{5.55cm} (x,\xi)\in \Omega^1, t=0 \\
& \lim_{|x|\to \infty}V = 0, \hspace{4.15cm}\qquad (x,\xi)\in \Omega^1, t>0,
\end{cases}
\label{eq: PDEV1d}
\end{equation}
where $\tilde{H}(\xi-\xi_c) = \lim_{\dt \to 0} \frac{1}{\dt}H(\xi-\xi_c) =
\begin{cases}
+\infty &: \text{ if } \xi \geq \xi_c \\
0 &: \text{ otherwise}
\end{cases}$ is the result of taking the continuum limit.  
We can then calculate the spatially-variable, total density of cells in the live state as
\begin{equation}
p(x,t) = \int_0^\infty V(\x,t,\xi)\,d\xi.
\end{equation}
This method offers us the ability to develop more complex models, such as when $\beta$ can be negative, when the transition is stochastic (in which case we would substitute the Heaviside function with a probability distribution), or when we are interested in modeling agents in the secondary state.} \cso{In this case, we can more easily separate different populations and capture additional state change driven dynamics (e.g. where movement rules could depend on the given state).}

\cso{The alternate approach, which is the focus of this paper, is suitable in the case where we are only interested in the live agents and  there is a single deterministic state transition (when $\xi>\xi_c$) and $\beta>0$. In this case,} \cmy{it is possible to calculate the spatially-variable, total density, $p$, directly from the absorption model, $U$ as}
\begin{equation}
p(\x,t) = \int_0^{\xi_c} U(\x,t,\xi)\,d\xi.
\label{eq: livestate}
\end{equation}
If we initialize $\int_{\Omega^n}U(\x,0,\xi)\,d\xi\,d\x = 1$, then we can consider $p(\x,t)$ the probability that an agent is at location $\x$ at time $t$ and in the initial live state\footnote{If we want to find the probability an agent is at a particular location at a given time, given that the agent is in the initial live state, we can calculate:
\[
\prob(\x,t) = 
\frac{\int_0^{\xi_c}U(\x,t,\xi)\,d\xi}
{\int_0^\infty U(\x,t,\xi)\,d\xi}
= \frac{p(\x,t)}{\int_0^\infty U(\x,t,\xi)\,d\xi}.
\]
}.

We can rewrite \eqref{eq: PDE} as a PDE of $p$.  Let us integrate the terms from 0 to $\xi_c$ with respect to $\xi$.  This gives us
\[
\int_0^{\xi_c}U_t \,d\xi
+ \int_0^{\xi_c}\beta(\x)U_\xi \,d\xi
=
\int_0^{\xi_c}D_n \nabla^2 U\,d\xi.
\]
Given $U\in L^1(\Omega^n)$, we can switch derivatives and integrals using Fubini's theorem.  The above system reduces to a non-homogeneous diffusion equation
\begin{equation}
\begin{cases}
& p_t - D_n \nabla^2 p = f(\x,t) , \qquad \x\in \R^n, t>0 \\
& p(\x,0) = g(\x) , \hspace{2.15cm} \x\in \R^n, t=0\\
& \lim_{|\x|\to \infty} p = 0, \hspace{2.1cm} \x\in \R^n, t>0,
\end{cases}
\end{equation}
where $f(\x,t) = -\beta(\x)U\Big|_{\xi=\xi_c}$ and $g(\x) = \int_0^{\xi_c}\phi(\x,\xi)$.  
If we know the value of $f$, then we have an explicit solution for $p$ using the method of Green's functions (fundamental solutions).
In most cases we will not have the explicit value of $f(\x,t)$, in which case we must first solve for $U(\x,t,\xi)$ before integrating to compute $p(\x,t)$.

We can use the value of $p$ to calculate cellular properties of interest, such as flux out of the initial live state or the average time in the initial live state.

\subsection{Mean occupancy time}
We may be interested in the mean time an agent is in the initial live state, which is denoted as the mean occupancy time (MOT).  
In a manner similar to deriving the mean first passage time, this is the first moment of the total flux out of a particular state.

The total flux out of the initial state can be computed as
\[
F(\x,t) = -\pder[]{t}\int_{\R^n}p(\x,t)\,d\x.
\]
The negative sign is due to the fact that we are tracking the density exiting the initial state.  It follows that the MOT is
\begin{equation}
M = \int_0^\infty t F(\x,t)\,dt.
\end{equation}
Since $p\in L^1(\R^n)$ and for any finite location $\x\in \R^n$, $\lim_{t\to\infty}p(\x,t)=0$, we can use integration by parts to derive the MOT,
\begin{equation}
M = \int_0^\infty \int_{\R^n} p(\x,t)\,dx\,dt.
\end{equation}

\section{Mathematical analysis}\label{sec:analysis}
Through the derivation of this continuous approximation, higher order terms in the Taylor series expansions were neglected. We must still ensure that we are maintaining the proper physics with this new equation. For example, we wish that energy in the system is not increasing and that the total quantity of agents or cells is conserved. In addition, since the governing equation \eqref{eq: PDE} is classified as a mixed Parabolic-Hyperbolic PDE,  there is no generalized theorem we can apply to show it is well-posed.
To this end, Theorems 3-5 in this section will prove existence, uniqueness, and continuous dependence on initial data, respectively.

\subsection{Energy \& conservation}\label{sec: Conservation}
In order to prove uniqueness and the continuous dependence of the PDE solution on initial data, we need to show that there is some time-dependent functional $E(t)$, such that our solution $U$ of \eqref{eq: PDE} satisfies $0\leq E(t) \leq E(0)$ for all $t>0$.  We will refer to this functional as the energy of the solution at time $t$.  To match the physics of the ABM simulation, the energy of our PDE should be non-increasing, but we need to prove that our PDE does not lose this feature during the process of deriving the continuum approximation.
\begin{theorem}
Suppose $\beta(\x)>0$ for all $\x\in \R^n$.  
The PDE \eqref{eq: PDE} with the energy functional\hspace{.08cm}\footnote{The energy functional is not meant to be interpreted as physical energy (i.e. it is not kinetic or potential energy), it is a naming convention that is consistent with classical PDE theory \cite{Evans10}.}
$E(t) = \frac{1}{2} \norm[2]{ U }$
satisfies the inequality
$0 \leq E(t) \leq E(0)$.
\label{thm: Energy}
\end{theorem}

\begin{proof}
Via a calculation,
\begin{align*}
\frac{dE}{dt} &= \int_{\Omega^n} U U_t\,d\xi\,d\x
=
 \int_{\Omega^n} U[-\beta(\x)U_\xi + \cmy{D_n} \nabla^2 U]\,d\xi\,d\x
\\
&=
- \int_{\Omega^n} \beta(\x)UU_\xi\,d\xi\,d\x
+ D_n  \int_{\Omega^n} U\nabla^2 U\,d\xi\,d\x.
\end{align*}
First, integration by parts in the variable $\xi$ gives
\[
 \int_{\Omega^n} \beta(\x)UU_\xi \,d\xi\,d\x = 
\int_{\R^n} \left[\beta(\x)U^2\Big|_{\xi=0}^\infty\right]\,d\x -
 \int_{\Omega^n} \beta(\x) UU_\xi\,d\xi\,d\x.
\]
We assume that for any finite $t>0$ that $U=0$ as $\xi\to \infty$.  Given $\beta(\x)>0$, then $U=0$ at $\xi=0$ for any $t>0$.  Thus, $ \int_{\Omega^n} \beta(\x)UU_\xi \,d\xi\,d\x = 0$.
Second, by the Divergence product rule,
\[
\int_{\R^n} |\nabla U|^2 \,d\x = \int_{\partial \R^n}U\nabla U \cdot \hat{\eta}\,d\x - \int_{\R^n} U\nabla^2 U\,d\x,
\]
where $\hat{\eta}$ is the unit outward normal vector.
Considering $\lim_{|\x|\to \infty}U = 0$, we have that
\[
 \int_{\Omega^n} U\nabla^2 U\,d\xi\,d\x = 
- \int_{\Omega^n} |\nabla U|^2\,d\xi\,d\x.
\]
Therefore, we have that for every $t>0$,
\[
\frac{dE}{dt} = -D_n  \int_{\Omega^n}|\nabla U|^2\,d\xi\,d\x \leq 0.
\]
Seeing that $\frac{dE}{dt}\leq 0$, we have $0\leq E(t) \leq E(0)$.

\end{proof}
\cmy{Note that the energy functional $E_V(t) = \frac{1}{2} \norm[2]{V}$ also satisfies the inequality $0\leq E_V(t) \leq E_V(0)$ since for every $t>0$,
$\int_{\Omega^n}\tilde{H}(\xi-\xi_c)U^2\,d\xi\,d\x \geq 0$.}

In the ABM simulation, no agent is removed from the system.  Again, we want the PDE solution to match the important physics of the ABM simulation.  We do so by proving that the solution $U$ is conserved at each time $t$ over the entire domain $\Omega^n$.
\begin{theorem}(Conservation)
Suppose $U\in L^1(\Omega^n)$ solves \eqref{eq: PDE} and $\beta(\x)>0$ for all $\x\in \R^n$.  Then $\int_{\Omega^n} U\,d\xi\,d\x = \int_{\Omega^n}\phi(\x,\xi)\,d\xi\,d\x$ for any $t>0$.
\end{theorem}
\begin{proof}
By means of a calculation,
\begin{align*}
\pder[]{t}\int_{\Omega^n} U \,d\xi\,d\x
&= \int_{\Omega^n} \pder[U]{t} \,d\xi\,d\x
= \int_{\Omega^n} \left\{ D_n \nabla^2 U - \beta(\x)U_\xi \right\} \,d\xi\,d\x
\\
&= D_n \int_{\partial \Omega^n}  \nabla U \cdot \hat{\eta} \,dS
- \int_{\R^n} \beta(\x)\left[ U\Big|_{\xi=0}^{\infty} \right]\,d\x.
\end{align*}
Since $\lim_{|\x|\to\infty}U=0$ we have that the first term is 0.
Also, given $\beta(\x)>0$ for all $x\in\R^n$, then for $t>0$ we have $U(\x,t,\xi=0)=0$, and the second term is also 0.  
It follows that $\pder[]{t}\int_{\Omega^n} U\,d\xi\,d\x = 0$.
Therefore, $U$ is conserved.
\end{proof}

\subsection{Operator-splitting semi-discrete solution}

We will approximate a solution to the PDE in \eqref{eq: PDE} by splitting the linear operator and then solving the resulting system iteratively.  This gives us a solution that is discrete in time and continuous in spatial and absorption dimensions. 
We will first derive this semi-discrete solution and then show that it is well-posed.

Let $U = \hU(\x,t|\xi)  \bU(\xi,t|\x)$, where $\hU$ leaves $\xi$ fixed and $\bU$ leaves $\x$ fixed.  We can see that
$\bU\hU_t + \hU\bU_t + \beta(\x)\hU\bU_\xi = D_n \bU\nabla^2 \hU$ and
it follows that
$\bU \left(
\hU_t - D_n \nabla^2 \hU
\right) +
\hU \left( 
\bU_t + \beta(\x)\bU_\xi
\right)
=0.$
Assuming that $\bU$ and $\hU$ are not identically 0, we can then solve the following PDEs
\begin{equation}
\begin{cases}
& \hU_t - D_n \nabla^2 \hU = 0, \qquad \x\in \R^n, t>0 \\
& \hU = \hat{\phi}(\x|\xi), \hspace{1.9cm} \x\in \R^n, t=0 \\
& \lim_{|\x|\to \infty}\hU = 0, \hspace{1.25cm} \x\in \R^n, t>0,
\end{cases}
\label{eq: spatial}
\end{equation}
\begin{equation}
\begin{cases}
& \bU_t - \beta(\x) \bU_\xi = 0, \qquad \xi \in [0,\infty), t>0 \\
& \bU = \overline{\phi}(\xi|\x), \hspace{1.85cm} \xi \in [0,\infty), t=0.
\end{cases}
\label{eq: absorption}
\end{equation}

We solve the system in \eqref{eq: spatial} using the method of Green's functions\footnote{
The Green's function used in this paper is the fundamental solution of the diffusion equation.  However, for other spatial domains, such as a half plane or disk, we can use the method of images with the fundamental solution to derive an appropriate Green's function.} 
and convoluting with the initial condition:
\begin{equation}
\hU = G(\x,t) * \hat{\phi}(\x|\xi), \quad \forall \xi \geq 0, t>0,
\label{eq: SolveGrn}
\end{equation}
where 
\begin{equation}
G(\x,t) = \frac{1}{(4\pi D_n t)^{n/2}}\exp\left\{-\frac{|\x|^2}{4D_n t}\right\}, \qquad t>0,
\label{eq: FundamentalSol}
\end{equation}
is the fundamental solution of the diffusion equation in $\R^n$.
We solve \eqref{eq: absorption} using the method of characteristics:
\begin{equation}
\bU = \overline{\phi}(\xi-\beta(\x)t | \x), \quad
\forall \x \in \R^n, \qquad t>0.
\label{eq: SolveAbs}
\end{equation}

Our solution of \eqref{eq: PDE} alternates between \eqref{eq: SolveGrn} and \eqref{eq: SolveAbs} as the solution marches forward in time.  As we are not solving the system simultaneously, we will choose a length of time, $0< \tau \ll 1$, in which each solution is valid.  We will denote the solution at time $t=m\tau$ as $U^m (\x,\xi)$. The following iterative algorithm solves the semi-discrete, operator splitting system:
\begin{itemize}
\item Initialize $U^0(\x,\xi) = \phi(\x,\xi)$
\item For $m=1,2,\ldots$:
\begin{itemize}
\item[$\lozenge$] $\bU^{m-1}(\x|\xi) = U^{m-1}(\x,\xi)$
\item[$\lozenge$] $\hU^{m}(\xi|\x) = \bU^{m-1}(\x|\xi-\beta(\x)\tau)$
\item[$\lozenge$] $U^{m}(\x,\xi) = 
G(\x,\tau) * \hU^{m}(\xi|\x)$
\end{itemize}
\end{itemize}
Combining these solutions gives us the recurrence relation for the semi-discrete solution with time step $\tau$:
\begin{equation}
U^{m+1}(\x,\xi) = G(\x,\tau) * U^{m}(\x,\xi-\beta(\x)\tau).
\label{eq: semidiscrete}
\end{equation}
Additionally, we can use our recurrence relation to rewrite the solution at $t=m\tau$ in terms of the initial condition $\phi(\x,\xi)$, given as
\begin{equation}
U^m(\x,\xi) = 
G(\x,\tau)*^m \phi(\x,\xi-\beta(\x)m\tau) , \quad \forall (\x,t)\in\Omega^n. 
\label{eq: semidiscreteiterate}
\end{equation}

\subsection{Existence}\label{sec: Existence}
Our semi-discrete solution for $U^m(\x,\xi)$, given in \eqref{eq: semidiscreteiterate}, depends on the recurrence time step $\tau$ and the number of iterations $m$.  So, we define 
\begin{equation}
z_{m,\tau}(\x,\xi) \equiv G(\x,\tau)*^m \phi(\x,\xi-\beta(\x)m\tau) 
\label{eq: semidiscretez}
\end{equation}
as the approximation\cmy{\footnote{
\cmy{We could also make a similar proof for approximating V with 
$z_{m,\tau}(\x,\xi) \equiv G(\x,\tau)*^m \left[\phi(\x,\xi-\beta(\x)m\tau)\mathbf{1}_{[0,\xi_c)}(\xi-\beta(\x)m\tau - \xi_c)\right]$.  The indicator function is a result of solving the same operator splitting system as U with the additional equation
\[
\begin{cases}
\tilde{V}_t = -\tilde{H}(\xi-\xi_c)\tilde{V}, \quad &\x\in \R^n, \xi \in [0,\infty), t>0 \\
\tilde{V} = \bar{V}, \quad &\x\in \R^n, \xi \in [0,\infty), t=0. 
\end{cases}
\]}
}} 
of $U(\x,m\tau, \xi)$, accounting for the choices of both $\tau$ and $m$ (the same notation as used in \cite{Blank05}).
We want to show that the $L^1(\Omega^n)$ limit, $\mathcal{U}(\x,t,\xi)$,  of the sequence $\left\{z_{m,t/m}\right\}_{m \in \N}$ exists.  That is, the limit of the recurrence relation for a time $t$ exists when the recurrence time step, $\tau=t/m$, approaches 0.  
If this limit exists, it proves the existence of a solution, $U(\x,t,\xi)$, to the governing PDE, given in \eqref{eq: PDE}.  

For the $L^1$ limit to make sense, we first need to show that $z_{m,\tau}\in L^1(\Omega^n)$.
\begin{lemma}
Suppose $\phi \in L^1(\Omega^n) $ and 
$z_{m,\tau}(\x,\xi) = G(\x,\tau)*^m \phi(\x,\xi-\beta(\x)m\tau)$ for all $m\in \N, \tau>0$.  
Then $z_{m,\tau}(\x,\xi) \in L^1(\Omega^n)$, for all $\x\in \R^n, \xi\geq 0$.
\label{lemma: L1}
\end{lemma}
\begin{proof}
We know that $G(\x,\tau)\in L^1(\Omega^n)$.  By reason that $L^1(\Omega^n)$ is closed under convolution, we have that $G(\x,\tau)*^m \phi(\x,\xi-\beta(\x)m\tau) \in L^1(\Omega^n)$.

\end{proof}

\begin{lemma}
For any $t>0$, $G(\x,t/m)*^{m}\delta(\xi) \to \delta (\x,\xi)$ as $m \to \infty$ in $L^1(\Omega^n)$.
\label{lemma: Convolution}
\end{lemma}
\begin{proof}
We know that $\lim_{m\to \infty}\int_{\R^n} \left|G(\x,t/m)- \delta(\x)\right|\,d\x=0$, so it follows that
$\lim_{m\to \infty}
\norm[1]{
G(\x,t/m)*\delta(\xi)- \delta(\x,\xi)}
=0$.
First, we will show, via induction, that  $\lim_{m\to \infty}\norm[1]{
G(\x,t/m)*^h \delta(\xi) - \delta(\x,\xi)}
=0$ 
for any $h \in \N$.

As a base case, we will show that $\lim_{m\to \infty}\norm[1]{
(G*G)(\x,t/m)*\delta(\xi) - \delta(\x,\xi)}
=0$.  
By the Dominated Convergence theorem, we can see that
\begin{align*}
\lim_{m\to \infty}G*G(\x,t/m)
&=
\int_{\R^n} \left( \lim_{m\to \infty}G(\x-\y,t/m) \right)
\left(\lim_{m\to \infty} G(\y,t/m)\right)\,d\y
\\
&=
\int_{\R^n} \delta(\x-\y)\delta(\x)\,d\y
=
\delta(\x).
\end{align*}
By calculation,
{\footnotesize
\begin{align*}
\lim_{m\to \infty}\norm[1]{
(G*G)(\x,t/m)*\delta(\xi) - \delta(\x,\xi)}
&= 
\norm[1]{\lim_{m\to \infty}
(G*G)(\x,t/m)*\delta(\xi) - \delta(\x,\xi)}
\\
&= 
\norm[1]{
\delta(\x,\xi)-\delta(\x,\xi) } 
= 0.
\end{align*}
}
Now, we can assume that 
$\lim_{m\to \infty}\norm[1]{
G(\x,t/m)*^h \delta(\xi) - \delta(\x,\xi)}
=0$ 
for some $h\in \N$.  
We can calculate that
{\footnotesize
\begin{align*}
\lim_{m\to \infty}\norm[1]{
G*[G*^h (\x,t/m)\delta(\xi)] - \delta(\x,\xi)}
&= 
\norm[1]{ \lim_{m\to \infty}
G*[G(\x,t/m)*^h \delta(\xi)] - \delta(\x,\xi)
}
\\
&= 
\norm[1]{ \delta(\x,\xi)-\delta(\x,\xi) }
= 0
\end{align*}
}
by the inductive assumption.
It follows that
{\footnotesize
\begin{align*}
\lim_{m\to \infty}
\norm[1]{ G(\x,t/m)*^{h+1} \delta(\xi) - \delta(\x,\xi)
}
&= \lim_{m\to \infty}\norm[1]{
G * [G(\x,t/m)*^{h}\delta(\xi)] - \delta(\x,\xi)
}
\\
& = 
\norm[1]{ \lim_{m\to \infty} 
G * [G(\x,t/m)*^{h}\delta(\xi)] - 
\delta(\x,\xi)
}
\\
& = \norm[1]{
\delta(\x,\xi) - \delta(\x,\xi) }
= 0
\end{align*}
}
for all $h\in \N$.

If we define $h = m$ for $m\in \N$.  Then 
$\lim_{m \to \infty}\norm[1]{
G(\x,t/m)*^{m}\delta(\xi) -\delta(\x,\xi)
}
= 0$. 
\end{proof}

\begin{lemma}
Suppose $\phi \in L^1(\Omega^n)$ and 
$z_{m,t/m}(\x,\xi) = G(\x,t/m)*^m \phi(\x,\xi-\beta(\x)t)$ for any $m\in \N, t>0$, as defined in \eqref{eq: semidiscretez}.  Then 
$\left\{ z_{m,t/m}\right\}_{m\in\N}$ 
is a Cauchy sequence in $L^1(\Omega^n)$.
\label{lemma: Cauchy}
\end{lemma}

\begin{proof}
We want to show
\[
\lim_{p,q\to \infty}\norm[1]{
z_{p,t/p}(\x,\xi)-z_{q,t/q}(\x,\xi) }
= 0.
\]
From Lemma \ref{lemma: Convolution} we know that, for any $m\in \N$,
$\lim_{m\to \infty}\norm[1]{
G(\x,t/m)*^{m}\delta(\xi)-\delta(\x,\xi) }
= 0$.  Then
{\footnotesize
\begin{align*}
\lim_{m\to \infty}\norm[1]{
z_{m,t/m}(&\x,\xi) - \phi(\x,\xi-\beta(\x)t) }
=
\\
&=
\lim_{m\to \infty}\norm[1]{
G(\x,t/m)*^{m}\phi(\x,\xi-\beta(\x)t) -
\delta(\x,\xi)*\phi(\x,\xi-\beta(x)t)}
\\
&\leq
\lim_{m\to \infty}\norm[1]{
G(\x,t/m)*^{m}\delta(\xi)-\delta(\x,\xi)
}\norm[1]{
\phi(\x,\xi-\beta(\x)t) }
= 0.
\end{align*}
}
It follows that
{\footnotesize
\begin{align*}
\lim_{p,q\to \infty}\norm[1]{
&z_{p,t/p}(\x,\xi)-z_{q,t/q}(\x,\xi) }
=
\\
&=
\lim_{p,q\to \infty}
\norm[1]{
z_{p,t/p}(\x,\xi)-\phi(\x,\xi-\beta(\x)t)+
\phi(\x,\xi-\beta(\x)t) -z_{q,t/q}(\x,\xi)
}
\\
&\leq
\lim_{p\to \infty}
\norm[1]{
z_{p,t/p}(\x,\xi)-
\phi(\x,\xi-\beta(\x)t) } 
+
\lim_{q\to \infty}\norm[1]{
z_{q,t/q}(\x,\xi)-
\phi(\x,\xi-\beta(\x)t)}.
\end{align*}
}
Therefore, 
$\lim_{p,q\to \infty}
\norm[1]{ 
z_{p,t/p}(\x,\xi)-
z_{q,t/q}(\x,\xi)}
= 0$.
\end{proof}

\begin{theorem} (Existence)
Suppose $\phi \in L^1(\Omega^n)$.  There exists a solution, $U\in L^1(\Omega^n)$, to the governing PDE:
\begin{equation}
\begin{cases}
&U_t + \beta(\x)U_\xi = D_n \nabla^2 U, \qquad (\x,\xi)\in \Omega^n, t>0 \\
& U(\x,t=0,\xi) = \phi(\x,\xi), \hspace{0.8cm} (\x,\xi)\in \Omega^n, t=0 \\
& \lim_{|\x|\to \infty}U(\x,t,\xi) = 0, \hspace{1cm} (\x,\xi)\in \Omega^n, t>0.
\end{cases}
\end{equation}
\label{thm: Existence}
\end{theorem}

\begin{proof}
Choose any $(\x,\xi)\in \Omega^n$ and any $t>0$.
Suppose $\phi \in L^1(\Omega^n)$ and 
define $z_{m,\tau}(\x,\xi) = G(\x,\tau)*^m \phi(\x,\xi-\beta(\x)m\tau)$ for any $m\in \N$, $\tau>0$.  
By Lemma \ref{lemma: L1} and Lemma \ref{lemma: Cauchy} we know that $z_{m,t/m}\in L^1(\Omega^n)$ is a Cauchy sequence.  
On account of $L^1$ being complete, there exists a $\mathcal{U}\in L^1(\Omega^n)$ such that 
$\lim_{m\to \infty}\norm[1]{
z_{m,t/m}(\x,\xi) - \mathcal{U}(\x,t,\xi)}=0$.
Since $z_{m,t/m}(\x,\xi)$ satisfies the operator-split PDE for all $m \in \N$, we know that $\mathcal{U}(\x,t,\xi)$ satisfies the operator-split PDE.
Therefore $\mathcal{U}(\x,t,\xi)$ satisfies the time-continuous PDE.
\end{proof}


\subsection{Uniqueness \& continuous dependence on initial data}\label{sec: Uniqueness}
\begin{theorem} (Uniqueness)
The solution to PDE \eqref{eq: PDE} is unique.
\label{thm: Uniqueness}
\end{theorem}
\begin{proof}
Suppose we have two solutions, $U_1, U_2 \in L^1(\Omega^n)$ to the PDE \eqref{eq: PDE}.  
We will define $W = U_2 - U_1$.  Given \eqref{eq: PDE} is linear, we know $W(\x,t,\xi)$ solves the PDE:
\begin{equation}
\begin{cases}
& W_t + \beta(\x)W_\xi = D_n \nabla^2 W, \qquad (\x,\xi)\in \Omega^n, t>0 \\
& W = 0, \hspace{4.0cm} (\x,\xi)\in \Omega^n, t=0 \\
& \lim_{|\x|\to\infty}W =0, \hspace{2.4cm} (\x,\xi)\in \Omega^n, t>0.
\end{cases}
\end{equation}
From the energy argument in Theorem \ref{thm: Energy}, we know that 
\[
0 \leq E_w (t) \leq E_w(0).
\]
Seeing that $E_w(0) = \frac{1}{2}\int_{\Omega^n}
W(\x,0,\xi)^2 \,d\xi\,d\x = 0$, 
we know that $E_w(t)=0$ for all $t$.  
By definition of $E_w(t)$, we demonstrated that
\[
0 \leq 
\int_{\Omega^n} \left(U_1(\x,t,\xi)-U_2(\x,t,\xi)\right)^2\,d\xi\,d\x =
E_w(t) = 0.
\]
Therefore, $U_1(\x,t,\xi)=U_2(\x,t,\xi)$ almost everywhere.
\end{proof}

\begin{theorem} (Continuous Dependence on Initial Data)
Consider any $\epsilon > 0$.  Suppose $U_1$ satisfies the PDE
\begin{equation}
\begin{cases}
& U_t + \beta(\x)U_\xi = D_n\nabla^2 U, \qquad (\x,\xi)\in \Omega^n, t>0 \\
& U(\x,t=0,\xi) = \phi_1(\x,\xi), \hspace{0.6cm} (\x,\xi)\in \Omega^n, t=0 \\
& \lim_{|\x|\to\infty}U(\x,t,\xi)=0, \hspace{0.93cm}(\x,\xi)\in \Omega^n, t>0,
\end{cases}
\label{eq: PDE1}
\end{equation}
and $U_2$ satisfies the PDE 
\begin{equation}
\begin{cases}
& U_t + \beta(\x)U_\xi = D_n\nabla^2 U, \qquad (\x,\xi)\in \Omega^n, t>0 \\
& U(\x,t=0,\xi) = \phi_2(\x,\xi), \hspace{0.6cm} (\x,\xi)\in \Omega^n, t=0 \\
& \lim_{|\x|\to\infty}U(\x,t,\xi)=0, \hspace{0.93cm}(\x,\xi)\in \Omega^n, t>0,
\end{cases}
\label{eq: PDE2}
\end{equation}
where $\norm[2]{\phi_1(\x,\xi)-\phi_2(\x,\xi)} < \epsilon$.  
Then $\norm[2]{U_1-U_2} < \epsilon$.
\label{thm: Continuous}
\end{theorem}

\begin{proof}
We will define $W = U_1-U_2$.  As \eqref{eq: PDE1} and \eqref{eq: PDE2} are both linear, $W$ solves the PDE
\begin{equation}
\begin{cases}
& W_t + \beta(\x)W_\xi = D_n \nabla^2 W, \qquad (\x,\xi)\in \Omega^n, t>0 \\
& W = \phi_1(\x,\xi)-\phi_2(\x,\xi), \hspace{0.95cm}(\x,\xi)\in \Omega^n, t=0 \\
& \lim_{|\x|\to\infty}W =0, \hspace{2.4cm}(\x,\xi)\in \Omega^n, t>0.
\end{cases}
\label{eq: PDE3}
\end{equation}
Let us define the energy of \eqref{eq: PDE3} as
\[
E_w(t) = \frac{1}{2}\int_{\Omega^n} W^2\,d\xi\,d\x 
= 
\frac{1}{2}\norm[2]{ W(t) }.
\]
By the same argument in the proof of Theorem \ref{thm: Energy}, we have that 
$0 \leq \norm[2]{W(\x,t,\xi) }
\leq \norm[2]{ W(\x,0,\xi) }$.
Since 
$\norm[2]{ W(\x,0,\xi) } =
\norm[2]{ \phi_1(\x,\xi)-\phi_2(\x,\xi) } 
< \epsilon$ 
and
$\norm[2]{ W(\x,t,\xi) } =
\norm[2]{ U_1(\x,,t,\xi)-U_2(\x,t,\xi) }$,
we have that
\[
0\leq 
\norm[2]{ U_1(\x,t,\xi)-U_2(\x,t,\xi) }
\leq
\norm[2]{ \phi_1(\x,\xi)-\phi_2(\x,\xi) }
< \epsilon.
\]
\end{proof}
\cmy{
From the energy argument in Theorem \ref{thm: Energy}, we know that the PDE for V also satisfies the properties of uniqueness and continuous dependence on initial data.  Thus, the model for V is also well-posed.
}

\section{Numerical approximation}\label{sec:numerics}

\subsection{Fully discrete derivation}
We are primarily interested in calculating 
\cmy{the spatially-variable, total density of agents in the live state, $p$.  
While demonstrating that the PDE models for $U$ and $V$ are both well-posed, we explained how the only difference in approximate solutions was multiplying by an indicator function $\mathbf{1}_{[0,\xi_c)}(\xi)$.  
This leads to the fact that the value of $p(\x,t)$ using either method will be identical.  
So for our numerical computation, we will compute the spatially-variable, total density of agents in the initial live state with} 
$p(\x,t) = \int_0^{\xi_c}U(\x,t,\xi)\,d\xi$.  
We will derive this numerical approximation within the spatial domain in 1-D, but the method can easily extend to higher dimensions.
Considering that we will first solve \eqref{eq: PDE} for $U$, we will discretize the region $\Omega^1$ using cell volumes (as opposed to discrete nodes).  
We divide the spatial component into $N$ bins\footnote{
The analytic solution requires the spatial domain to be $\R$.  However, numerically, we need to choose a finite domain.  We choose $N$ such that $G(N\,dx/2, \tau)$ is bounded close to 0.  Similarly, the absorption domain is $[0,\infty)$.  Our solution of interest is within the domain $[0,\xi_c)$ and hence, we choose $K$ such that $K\,d\xi$ is larger than $\xi_c$.}
of width $dx$ and the absorption component into $K$ bins of width $d\xi$;  
the cell volumes have area $dx d\xi$.

These cell volumes will be defined as 
$\omega_{i,k}=B(x_i,dx/2)\times [\xi_k,\xi_{k+1}]$, where $dx$ is the spatial discretization step-size and $B(x_i,dx/2) = \{y\in \R : |x_i-y|<dx/2\}$.  For the following derivations, the spatial location will be indexed by $i$ and the cumulative absorption amount will be indexed by $k$.
We then define $u_{i,k}^m \approx U^m(x_i,\xi_k)$ as
\begin{equation}
u_{i,k}^m = \frac{1}{dx\,d\xi}\int_{\omega_{i,k}} U^m(y,z)\,dy\,dz,
\end{equation}
the average value of $U^m$ in the cell volume $\omega_{i,k}$. 
Note that the continuous and semi-discrete solution is capitalized, $U(x,t,\xi)$ or $U^m(x,\xi)$, whereas the fully discrete solution is in lower-case, $u_{i,k}^m$.

We know the semi-discrete recurrence relation
$U^{m+1}(x,\xi) = G(x,\tau)*U^m(x,\xi-\beta(x)\tau)$ from equation \eqref{eq: semidiscrete}.  
This solution is fully discretized by integrating over the cell volume $\omega_{i,k}$.  By recalling that $u_{i,k}^m$ is piece-wise continuous over $\omega_{i,k}$, we can solve the convolution exactly with the approximated solution:

{\footnotesize
\begin{align*}
\int_{\omega_{i,k}}G(x,\tau)&*U^m(x,\xi-\beta(x)\tau)\,d\xi\,dx
= \int_{\omega_{i,k}} \int_\R G(y,\tau)U^m(x-y,\xi-\beta(x)\tau)\,dy\,d\xi \,dx
\\
&= \sum_{j\in \Z} \int_{B(x_j,dx/2)}\int_{B(x_i,dx/2)}\int_{\xi_k}^{\xi_{k+1}}G(y,\tau)U^m(x-y,\xi-\beta(x)\tau)\,d\xi\,dx\,dy 
\\
&= \sum_{j\in \Z} \int_{B(x_j,dx/2)} \left[
G(y,\tau) \int_{B(x_i,dx/2)}\int_{\xi_k}^{\xi_{k+1}}U^m(x-y,\xi-\beta(x)\tau)\,d\xi\,dx \right]\,dy
\\
&= \sum_{j\in \Z} dx\,d\xi\,u_{i-j,k}^m\int_{B(x_j,dx/2)}  G(y,\tau)\,dy.
\end{align*}
}
\noindent Since $u_{i,k}^{m+1} = \frac{1}{dx\,d\xi}\int_{\omega_{i,k}} U^m(y,z)\,dy\,dz$, we have
\[
u_{i,d_i}^{m+1} = 
\sum_{j\in \Z} u_{i-j,k}^m\int_{B(x_j,dx/2)}  G(y,\tau)\,dy,
\]
where $d_i = \lfloor k+\beta(x_i)\tau \rfloor$, the new absorption index.
By calculation, we find that 
\begin{equation}
\resizebox{0.9\textwidth}{!}{$
G_j = \int_{B(x_j,dx/2)}  G(y,\tau)\,dy =
\frac{1}{2}\left\{
erf\left(\frac{x_j+dx/2}{\sqrt{4D\tau}}\right)-
erf\left(\frac{x_j-dx/2}{\sqrt{4D\tau}}\right)
\right\}.
$}
\label{eq: GreenDiscrete}
\end{equation}
Our numerical method is then
\begin{equation}
u_{i,d_i}^{m+1} = 
\frac{1}{2}\sum_{j\in \Z}
u_{i-j,k}^m 
\left\{
erf\left(\frac{x_j+dx/2}{\sqrt{4D\tau}}\right)-
erf\left(\frac{x_j-dx/2}{\sqrt{4D\tau}}\right)
\right\}.
\end{equation}

We discretize the density $p(x,t)$ as
\begin{equation}
\resizebox{0.9\textwidth}{!}{$
p_i^m = \frac{1}{dx\,d\xi}\int_0^{\xi_c}\int_{B(x_i,dx/2)} U^m(y,z)\,dy\,dz \approx
\frac{1}{dx\,d\xi} \sum_{k\in \mathcal{A}} \int_{\omega_{i,k}}U^m(y,z)\,dy\,dz,\label{eq:pim}$}
\end{equation}
where $\mathcal{A} = \{k:kd\xi < \xi_c\}$.  
Therefore, we can represent $p$ numerically as $p_i^m = \sum_{k\in \mathcal{A}}u_{i,k}^m$, the exact integral using our piece-wise constant approximate solutions.

\subsection{Stability}
We can break down the numerical method into two steps: a diffusive step where we perform the convolution,
\[
v_{i,k}^{m+1} = \frac{1}{2}\sum_{j\in \Z}
u_{i-j,k}^m 
\left\{
erf\left(\frac{x_j+dx/2}{\sqrt{4D\tau}}\right)-
erf\left(\frac{x_j-dx/2}{\sqrt{4D\tau}}\right)
\right\}
\] 
and an absorption step $u_{i,d_i}^{m+1} = v_{i,k}^{m+1}$ where we change the indexing.
Further, we define the discrete energy functional as
\[
\mathbf{E}^m = \sum_{k=0}^K \sum_{i=0}^{N-1} \left(u_{i,k}^m\right)^2.
\]
To prove stability, we want to show $\mathbf{E}^{m+1}-\mathbf{E}^m \leq 0$.
Note that due to the indexing change, for any $i$, 
$\sum_{k=0}^K \left(u_{i,k}^{m+1}\right)^2 \leq \sum_{k\cmy{=}0}^K \left(v_{i,k}^{m+1}\right)^2$.

We can rewrite our numerical scheme as a matrix-vector product, $\mathbf{v}_k^{m+1} = \mathbf{u}_k^m * G = A\mathbf{u}_k^m$, where our discrete convolution matrix $A$ and vector indexing of $\mathbf{u}_k^m$ are the following
\begin{equation}
A = \begin{bmatrix}
G_0  & G_{N-1} & \hdots & G_2 & G_1 \\
G_1  & G_0    & G_{N-1} & \hdots & G_2 \\
\vdots & G_1 & G_0 & \ddots & \vdots \\
G_{N-2} &   & \ddots & \ddots & G_{N-1} \\
G_{N-1} & G_{N-2} & \hdots & G_1 & G_0
\end{bmatrix}
\quad
\mathbf{u}_k^{m} = 
\begin{bmatrix}
u_{0,k}^m \\ u_{1,k}^m \\ \vdots \\ u_{N-1,k}^m
\end{bmatrix},
\label{eq: matrixNumerical}
\end{equation}
given our definition of $G_j$, as defined in \eqref{eq: GreenDiscrete}.

The difference between the energy functional at subsequent times, having absorbed $\xi \in [kd\xi, (k+1)d\xi]$ particles, is:
\begin{align*}
E_k^{m+1}-E_k^m
&= \sum_{i=0}^{N-1} \left(u_{i,k}^{m+1}\right)^2 - \sum_{i=0}^{N-1} \left(u_{i,k}^m\right)^2 
\\
&\leq \sum_{i=0}^{N-1} \left(v_{i,k}^{m+1}\right)^2 -  \sum_{i=0}^{N-1} \left(u_{i,k}^m\right)^2
\\
&=  \left(A\mathbf{u}_{k}^m\right)^2 - \left(\mathbf{u}_k^m\right)^T\mathbf{u}_k^m
\\
&= \left(\mathbf{u}_k^m\right)^T\left(A^TA-I\right)\mathbf{u}_k^m,
\end{align*}
where $E_k^m \equiv \sum_{i =0}^{N-1}\left(u_{i,k}^m\right)^2$.

\begin{theorem}
The spectrum of $A^TA-I$ is $\cmy{s} \left( A^TA-I \right)\leq 0$, with the matrix $A$ defined in \eqref{eq: matrixNumerical} and the scalars $G_j$ defined in \eqref{eq: GreenDiscrete}.
\end{theorem}
\begin{proof}
The discrete convolution matrix $A$ is a circulant matrix, so it has eigenvalues
\[
\lambda_j = G_0 + G_{N-1}\gamma_j + G_{N-2}\gamma_j^2 + \hdots + G_1\gamma_j^{N-1} = 
\sum_{\ell=0}^{N-1} G_\ell \gamma_j^{N-\ell},
\]
for $j=0,1,\ldots, N_1$, where $\gamma_j = \exp\left\{\frac{2\pi j}{N}\sqrt{-1}\right\}$ (the $N$-th root of unity).  It follows that the amplitude of the $j$-th eigenvalue is
\begin{equation}
|\lambda_j| 
= \left|\sum_{\ell=0}^{N-1} G_\ell \gamma_j^{N-\ell}\right|
\leq \sum_{\ell=0}^{N-1} \left|G_\ell\right| \left|\gamma_j^{N-\ell}\right|
= \sum_{\ell=0}^{N-1} \left|G_\ell \right|.
\end{equation}
Given $G_\ell \geq 0$ for all $\ell$, we have that $|\lambda_j| \leq \sum_{\ell=0}^{N-1} G_\ell$ for all $j$.  
Since 
\begin{align*}
G_\ell &= \frac{1}{2}\left\{
erf\left( \frac{x_\ell+dx/2}{\sqrt{4D\tau}} \right) -
erf\left( \frac{x_\ell-dx/2}{\sqrt{4D\tau}} \right) \right\}
\\
&= \frac{1}{2}\left\{
erf\left( \frac{x_\ell+dx/2}{\sqrt{4D\tau}} \right) -
erf\left( \frac{x_{\ell-1}+dx/2}{\sqrt{4D\tau}} \right) \right\},
\end{align*}
we have that 
\begin{equation}
|\lambda_j| \leq
\frac{1}{2}\left\{
erf\left( \frac{x_{N-1}+dx/2}{\sqrt{4D\tau}} \right) -
erf\left( \frac{x_0-dx/2}{\sqrt{4D\tau} }\right) \right\}
< 1.
\end{equation}
The strict inequality is due to $-1\leq erf(x)\leq 1$ for all $x$ and $N$ being finite.  
It follows that the eigenvalues of $A^TA$ are $|\lambda_j|^2<1$.  Therefore, the spectrum of $A^TA-I$ is $\cmy{s}(A^TA-I)<0$.
\end{proof}

Therefore, $E_k^{m+1}-E_k^m \leq 0$.  
Consequently, 
\[
\mathbf{E}^{m+1}-\mathbf{E}^m = \sum_{k=0}^K \left\{ E_k^{m+1}-E_k^m\right\}\leq 0, 
\]
which proves that
the numerical method is stable.

\section{Numerical results}\label{sec:results}
\subsection{The 1-dimensional model}
For our 1-dimensional simulations, we perform 100,000 realizations of the ABM with agent $\mathfrak{u}$ initialized at $x_0 = 0.5$.  The agent moves with spatial step size of $\dx=0.01$ and time step $\dt = \dx^2/2$.  
For the corresponding PDE model, we use the stable numerical algorithm detailed in Section \ref{sec:numerics} with a point source at $x_0=0.5$. 
We choose $N$ so that $G_0, G_{N-1} < \varepsilon_{mach}$ and we assign the numerical step sizes as $dx=\dx$, $dt=\dt$, and $d\xi = \xi_c/2000$.  In both the ABM and PDE model,
we define the agent absorption function as $\beta(x) = \alpha \int_{B(x,\dx/2)} C(x)\,dx$.  
The $\alpha$ parameter defines the permeability of the agent's membrane and for the following examples, we let $\alpha = 0.1$.

\subsubsection{Example 1: Max concentration at starting location}
For this example, the chemical concentration $C(x)=\frac{1}{1+10(x-0.5)^2}$ is symmetric and concave down around $x=0.5$.  A comparison of the ABM and our continuum PDE model is shown in Fig.~\ref{fig: RatlSurf}, for a critical or tolerance threshold of $\xi_c = 10\dx\dt$. The distribution of cells or agents in the initial live state are shown in color with time on the vertical axis and spatial location on the horizontal axis. The values on Fig.~\ref{fig: PDERatlSurf} at location $(x_i,t_m)$ are the numerical solutions $p_i^m$ from \eqref{eq:pim}, which are interpreted as the probability a cell is alive and located within region $B(x_i,\dx/2)$ at time $t$. Since the agents are all initialized at $x_o=0.5$, we observe a high density of cells close to this point for small time intervals. We note that Fig.~\ref{fig: PDERatlSurf} is smoother than \ref{fig: ABMRatlSurf} since it is a continuous approximation whereas the ABM has agents moving discretely either to the left or right at each time step.

\begin{figure}[H]
     \centering
     \subfloat[][ABM Simulation]{
     \includegraphics[width=0.4\textwidth]{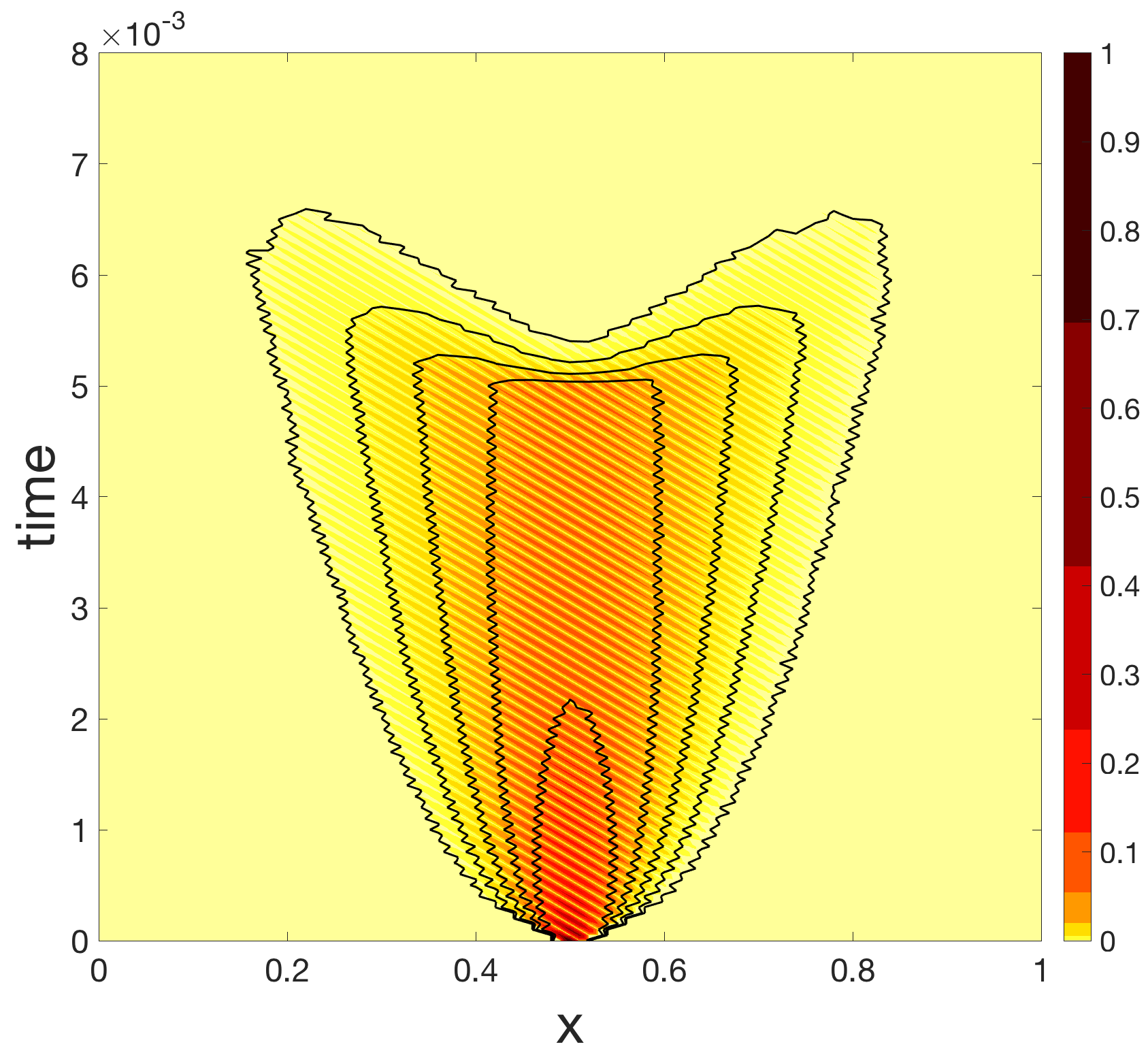}\label{fig: ABMRatlSurf}}
     \subfloat[][PDE Approximation]{
     \includegraphics[width=0.4\textwidth]{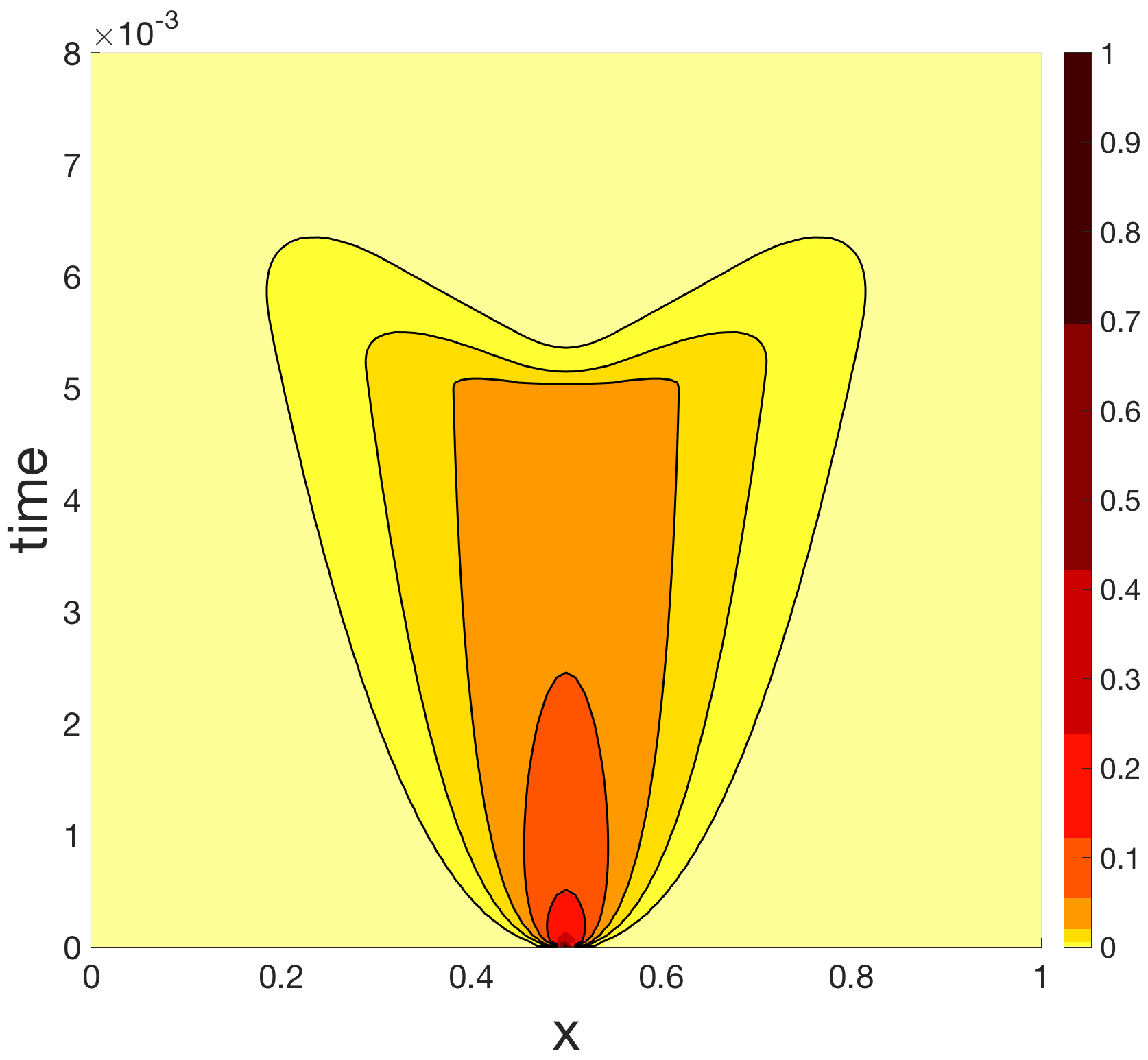}\label{fig: PDERatlSurf}}
     \caption{Comparison of the probability distribution of live agents (shown in color) at locations $x\in[0,1]$ and at time points $t\in[0,0.008]$. The ABM results are the mean over 100,000 simulations. \cso{For both, the chemical concentration is $C(x)=1/(1+10(x-0.5)^2)$.}}
     \label{fig: RatlSurf}
\end{figure}

Additionally, since $C(x)$ has a max at $x=0.5$, this causes the probability distribution $p(x,t)$ to become bimodal at approximately $t=0.0055$. Those cells that have remained close to the initial starting location have absorbed more particles than those that have moved left or right. Hence, cells close to $x=0.5$ are moving out of the initial live cell state when they reach their absorption capacitance $\xi_c$.

\begin{figure}[H]
     \centering
     \subfloat[][Survival Probability]{
     \includegraphics[width=0.30\textwidth,height=1.5in]{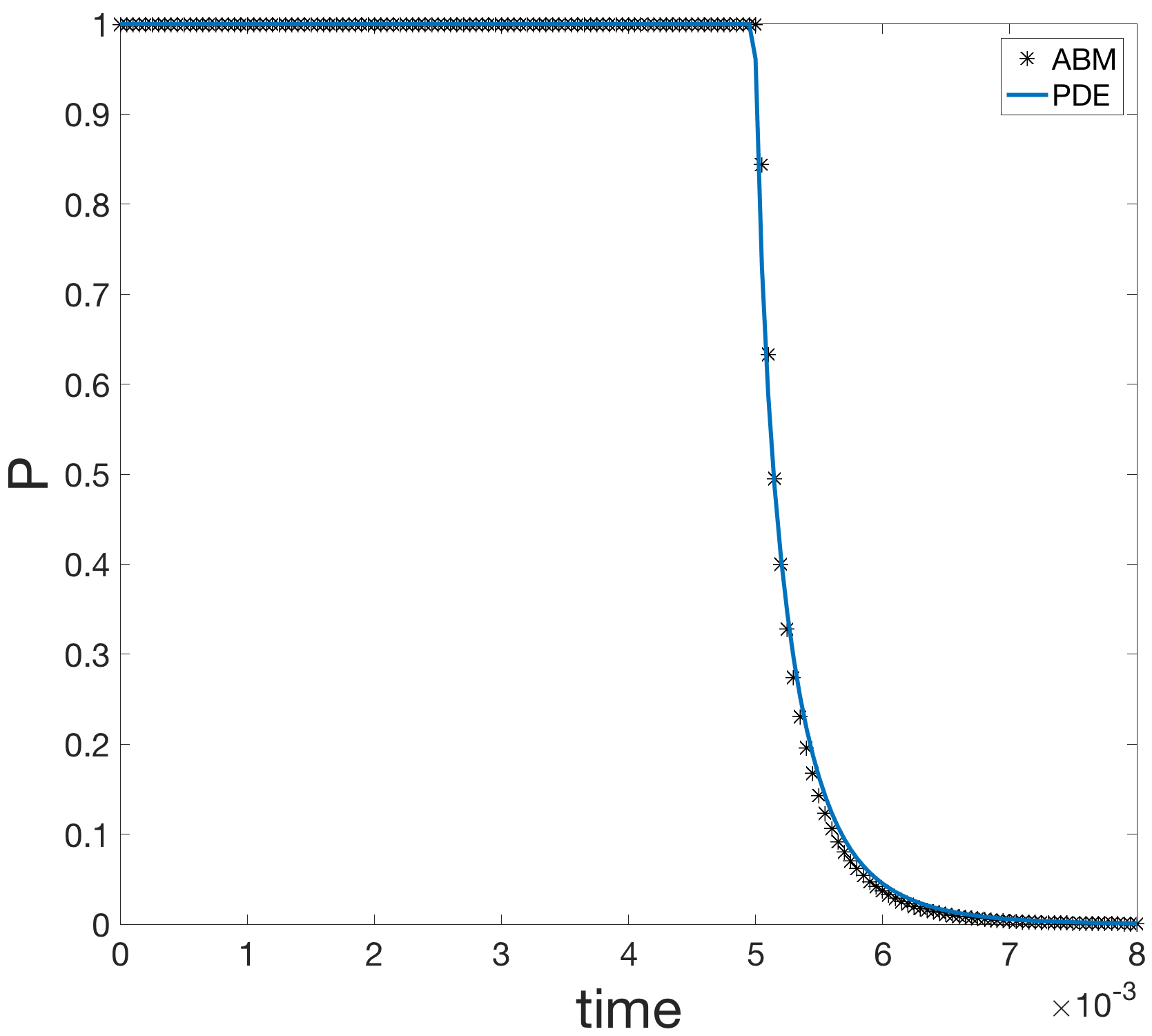}\label{fig: RatlSurvive}}
     \subfloat[][Mean Location]{
     \includegraphics[width=0.30\textwidth,height=1.55in]{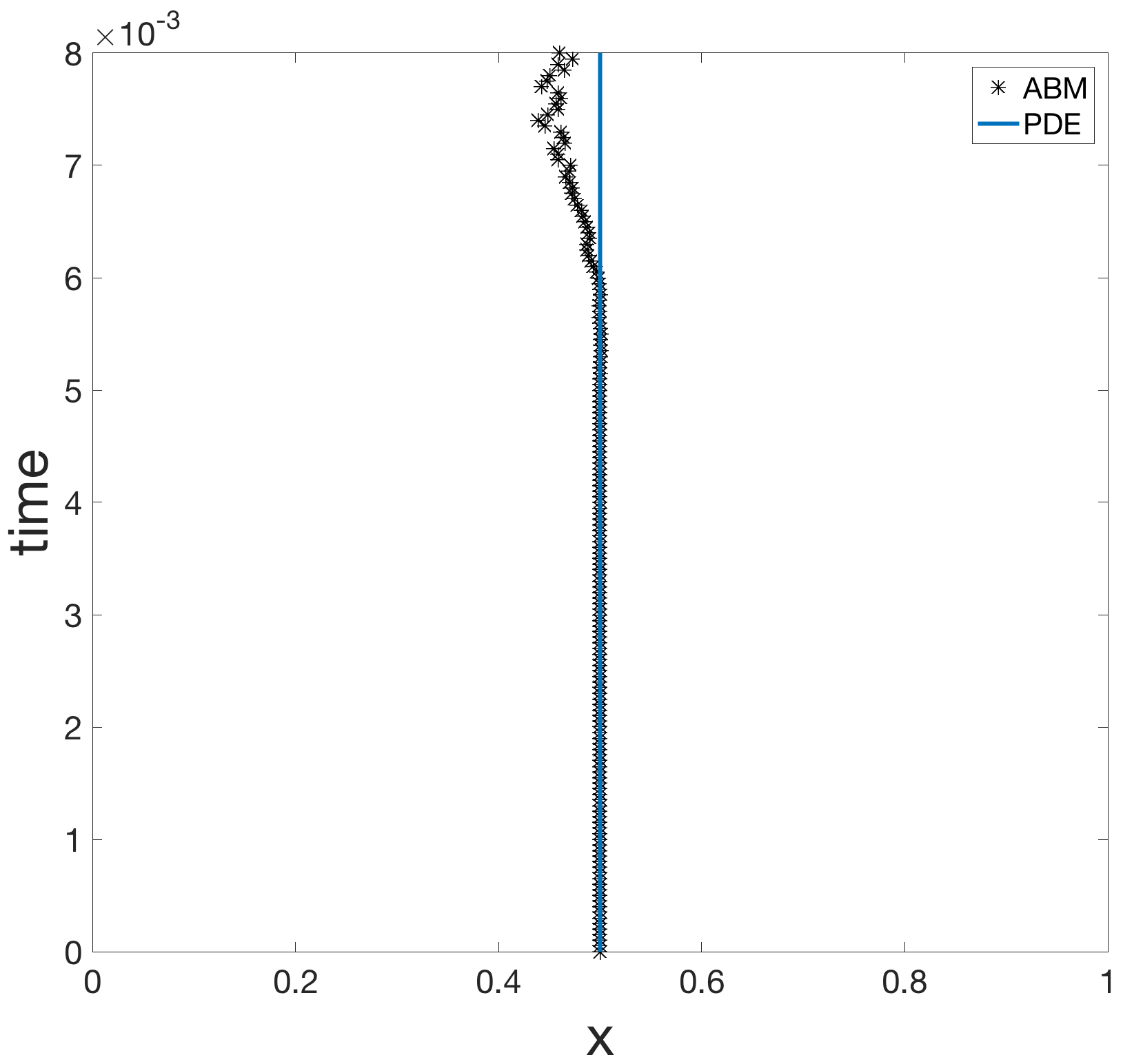}\label{fig: RatlMean}}
     \subfloat[][Standard Deviation]{
     \includegraphics[width=0.30\textwidth,height=1.5in]{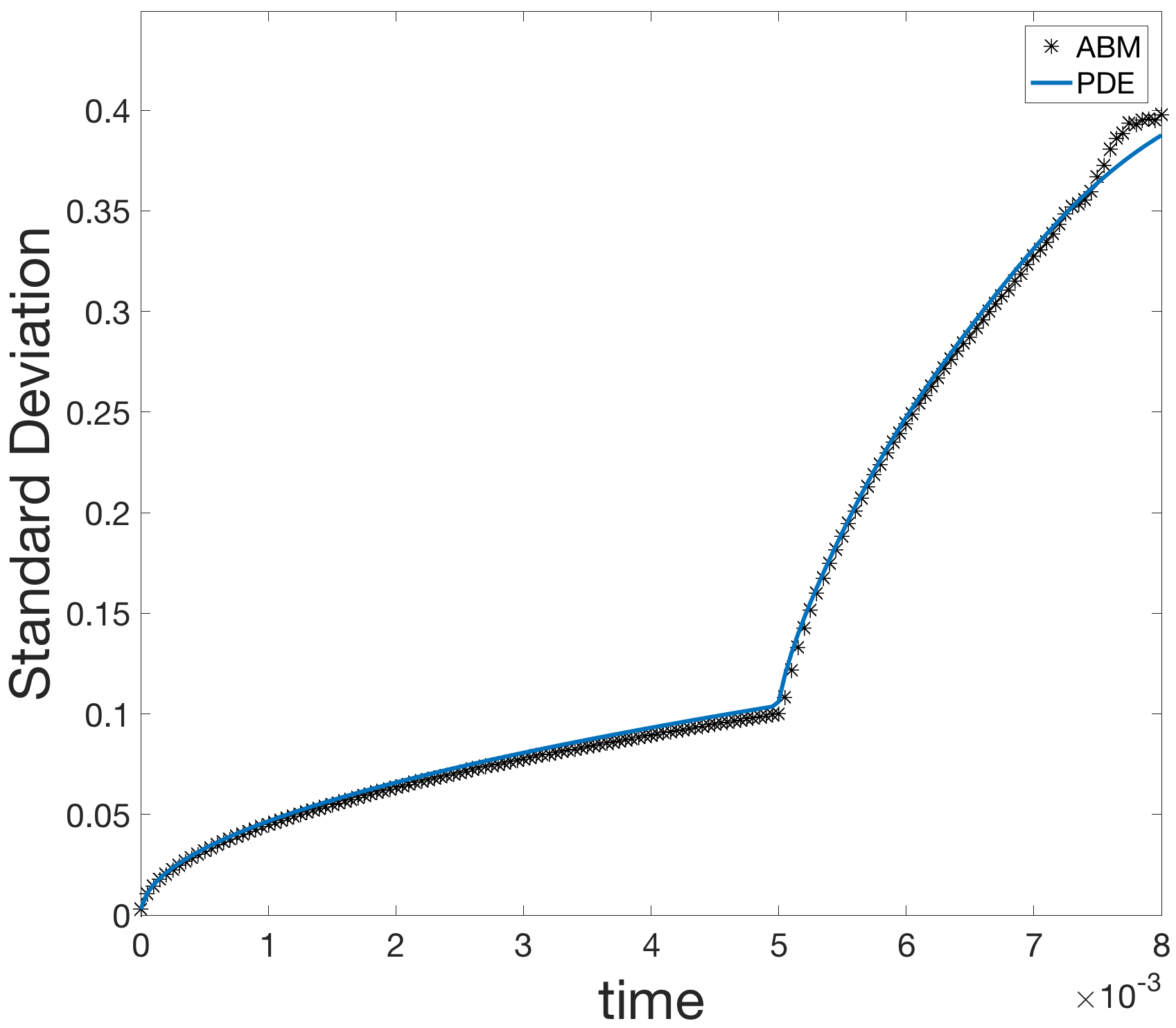}\label{fig: RatlStd}}
     \caption{Comparison of the survival probability, as well as mean and standard deviation of the live agent locations, for the ABM (black $*$) and the numerical PDE  solution (blue line) at each time-step \cso{for $C(x)=1/(1+10(x-0.5)^2)$}.}
     \label{fig: RatlData}
\end{figure}

The probability an agent is alive at a given time $t$ is the survival probability $P(t)$, calculated as 
\begin{equation}
P(t)=\int_\R p(x,t)\,dx.\label{eq:survivalP}
\end{equation} 
In Fig.~\ref{fig: RatlSurvive}, we observe that $P(t)$ for the ABM simulation and PDE approximations match; there is a sharp decrease in survival probability after $t=0.005$ and the majority of the cells have died at $t=0.007$.

The mean location of the live agents is calculated as $\mu(t) = \int_\R x\hat{p}(x,t)\,dx$, where $\hat{p}(x,t) = p(x,t)/P(t)$ is the normalized value of $p(x,t)$ at each time $t$.  
The numerical PDE solution solves for the average value in the interval centered at $x_i$ with radius $\dx/2$, $B(x_i,\dx/2)$.  
This allows the calculation of $\mu(t)$, the mean at time $t = m\tau$, as
\begin{equation}
\mu(t) = \int_\R x\hat{p}(x,t)\,dx = 
\frac{1}{P(t)}\sum_{i=1}^{N-1} p_i^m \int_{B(x_i,\dx/2)}x\,dx,\label{eq:mean}
\end{equation}
the exact integral of the approximate piece-wise constant solution.  Just as we did when calculating the convolution, we can take $p_i^m$ out of the integral since it is piece-wise constant.  
In a similar way, we can calculate $\sigma^2(t)$, the variance at time $t=m\tau$, as
\begin{equation}
\resizebox{0.9\textwidth}{!}{$
\sigma^2(t) = \int_\R (x-\mu(t))^2\hat{p}(x,t)\,dx =
\left\{ \frac{1}{P(t)} \sum_{i=1}^{N-1} p_i^m \int_{B(x_i,\dx/2)}x^2\,dx \right\} -
\mu(t)^2$.}
\end{equation}

The mean location of the ABM simulation and PDE approximation is shown in Fig.~\ref{fig: RatlMean}.  
The chemical concentration $C(x)$ is symmetric around $x=0.5$, the location where the agents are initialized, and there is no bias in movement ($\ell(x)=r(x)=0.5$). Hence, we would expect the mean location of agents in the initial state to be centered at $x=0.5$.  
We see that until approximately $t=0.006$, the PDE mean and the ABM mean are close to $x=0.5$.  
For times $t>0.006$, the number of agents in the ABM simulation is relatively small, as shown in Fig.~\ref{fig: RatlSurvive}.  
This accounts for the increasing stochastic noise in the mean, as well as the standard deviation, which is shown in Fig.~\ref{fig: RatlStd}. 

At each iteration of the ABM simulation, the agent can move either left or right.  We see that the agents that remain in the initial state are those that are furthest from $x=0.5$, where $C(x)$ is \cmy{smaller} than at $x=0.5$.  As a result, the standard deviation is a monotonically increasing function, as seen in Fig.~\ref{fig: RatlStd}.  At approximately $t=0.005$, many cells towards the center of the simulation change state, which causes the ``corner'' in the standard deviation graph.

\subsubsection{Example 2: Decreasing concentration}
The chemical concentration is $C(x)=\exp\left(-x^2\right)$, which is monotonically decreasing in the interval $[0,1]$ and all agents or cells are initialized at $x_o=0.5$.  
We expect that the agents which tend to move to the right within this interval have a higher probability of remaining in the initial state.  
As shown in Fig.~\ref{fig: ExpSurf}, the cells that remain in the initial state tend to be further to the right and again, we have excellent qualitative agreement between the ABM and the new PDE continuum model. In Fig.~\ref{fig: ABMExpSurf} we observe a striped pattern, which is a result of the ABM agents moving only left or right at any given iteration.  At a critical threshold of $\xi_c = 10\dx\dt$, cells are able to achieve a cumulative chemical absorption $\xi>\xi_c$, causing the cell to transition states or die. The survival probability shows this trend in Fig.~\ref{fig: ExpSurvive}, where there is a sharp decrease in survival probability at $t=0.055$.
\begin{figure}[h]
     \centering
     \subfloat[][ABM Simulation]{
     \includegraphics[width=0.4\textwidth]{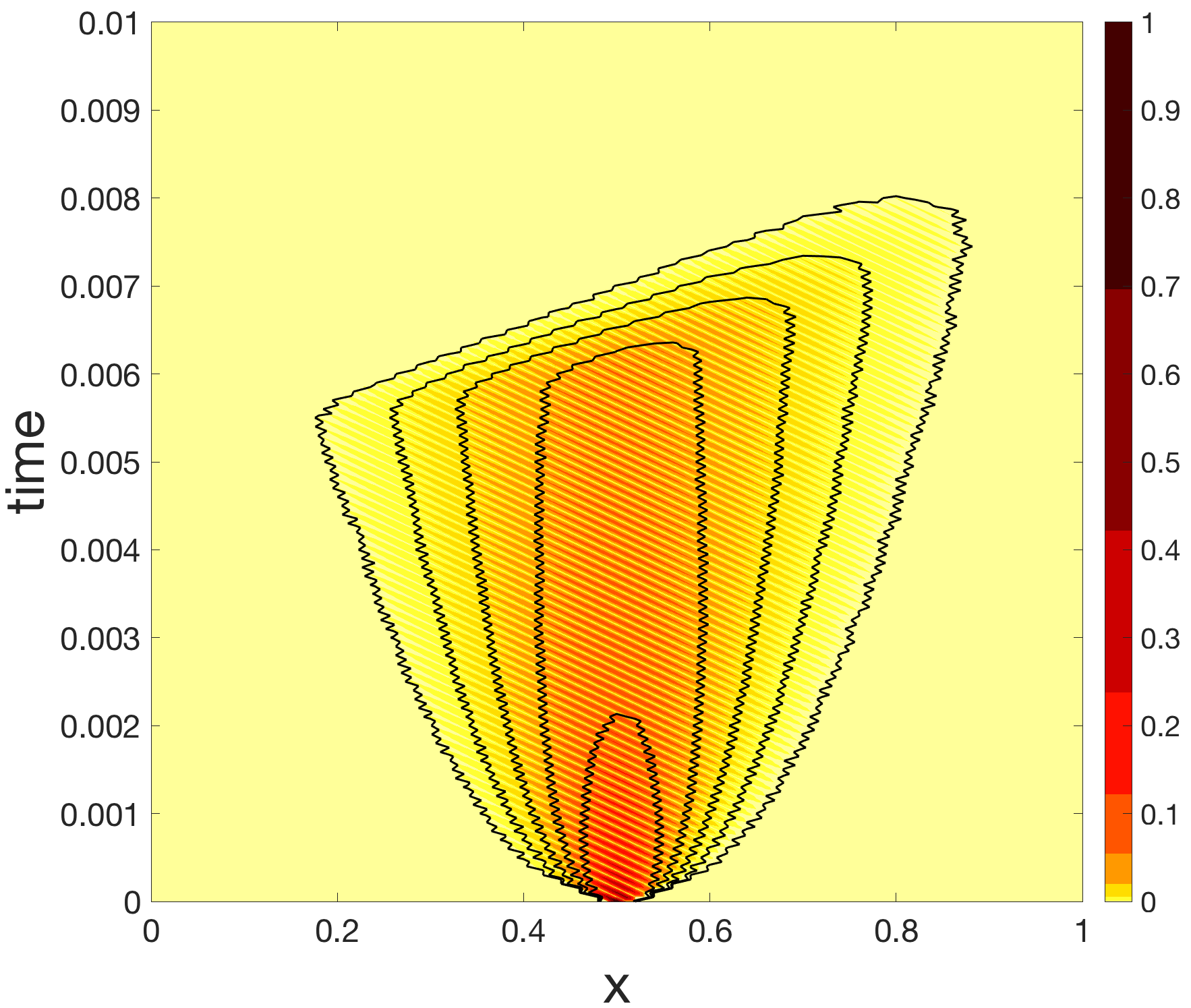}\label{fig: ABMExpSurf}}
     \subfloat[][PDE Approximation]{
     \includegraphics[width=0.4\textwidth]{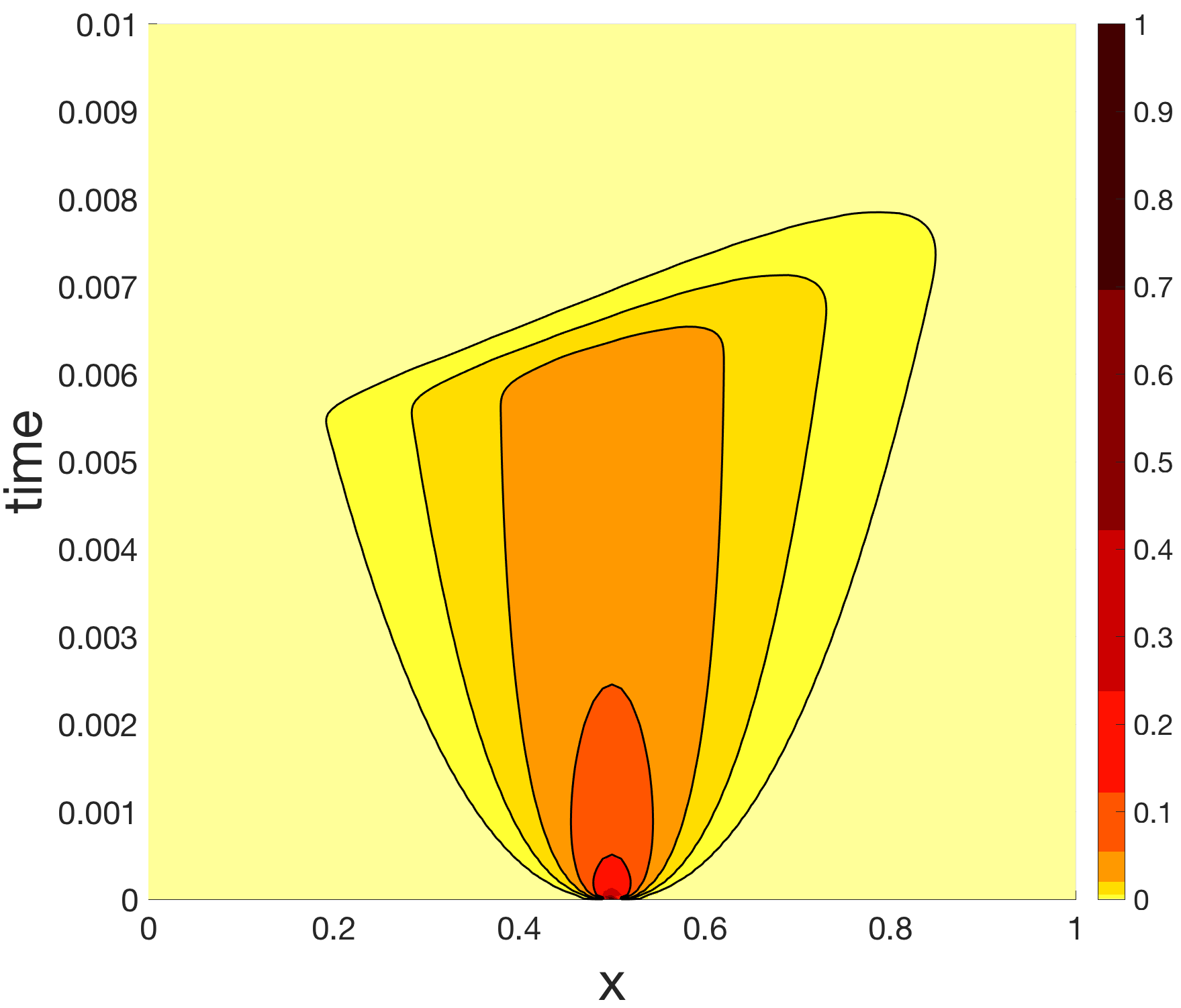}\label{fig: PDEExpSurf}}
     \caption{Comparison of the probability distribution of live agents (shown in color) at locations $x\in[0,1]$ and at time points $t\in[0,0.1]$. The ABM results are a mean of 100,000 simulations. \cso{For both, the chemical concentration is $C(x)=\exp(-x^2)$.}}
     \label{fig: ExpSurf}
\end{figure}

To further characterize the agreement between the ABM simulation and our PDE approximation, we again look at the mean and standard deviation of the location of live cells (with cumulative absorption $\xi<\xi_c$). In Fig.~\ref{fig: ExpMean}, we observe that the mean location (calculated using \eqref{eq:mean}) does move to the right of the initial location $x_o=0.5$ due to the decreased concentration $C(x)$ to the right of $x=0.5$ (allowing cells to live in this region for a longer period of time). Again, we see that there is noise in the ABM mean for times $t>0.008$, when there are relatively few agents in the initial state.

\begin{figure}[h]
     \centering
     \subfloat[][Survival Probability]{
     \includegraphics[width=0.3\textwidth,height=1.5in]{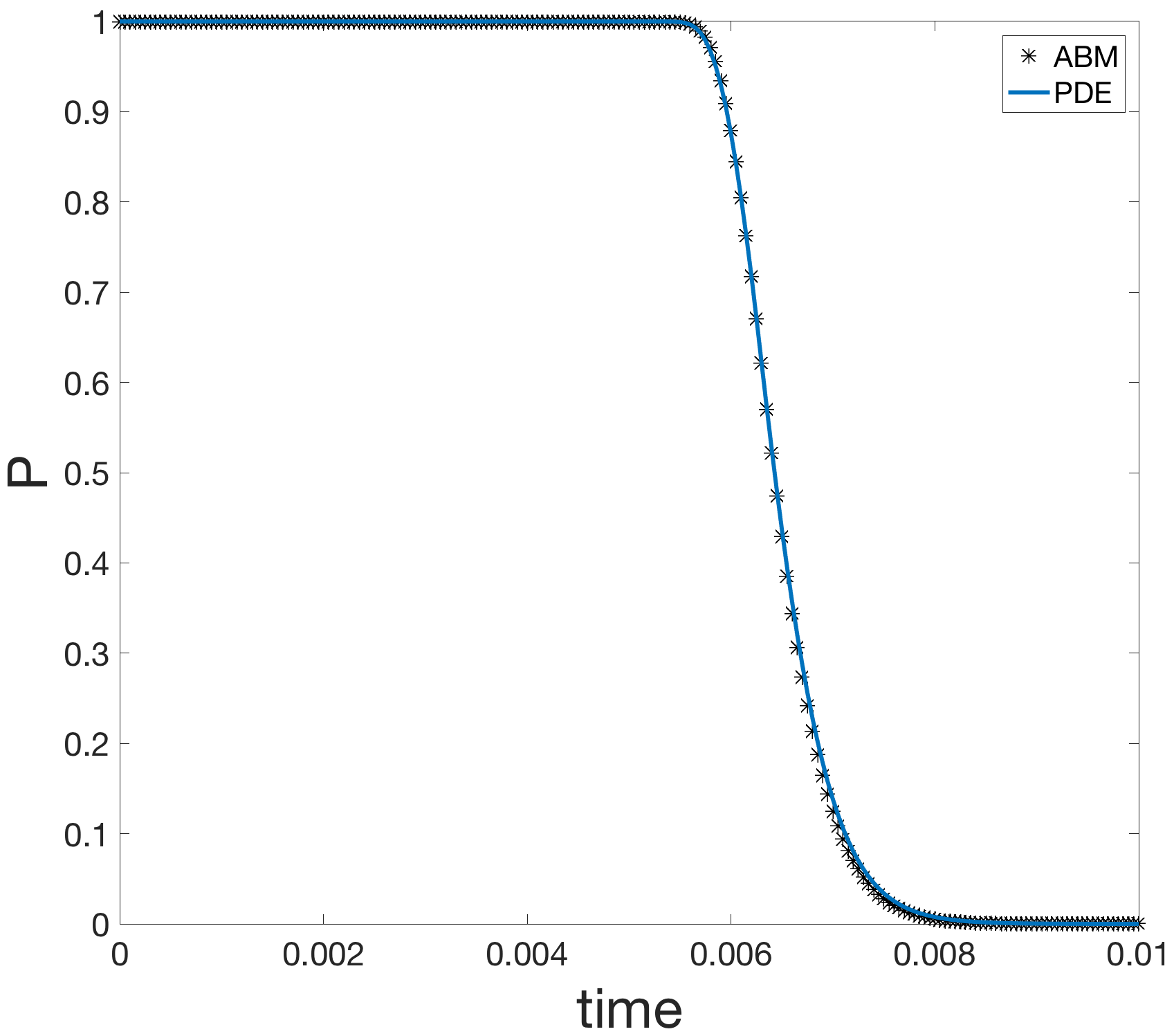}\label{fig: ExpSurvive}}
     \subfloat[][Mean Location]{
     \includegraphics[width=0.3\textwidth,height=1.51in]{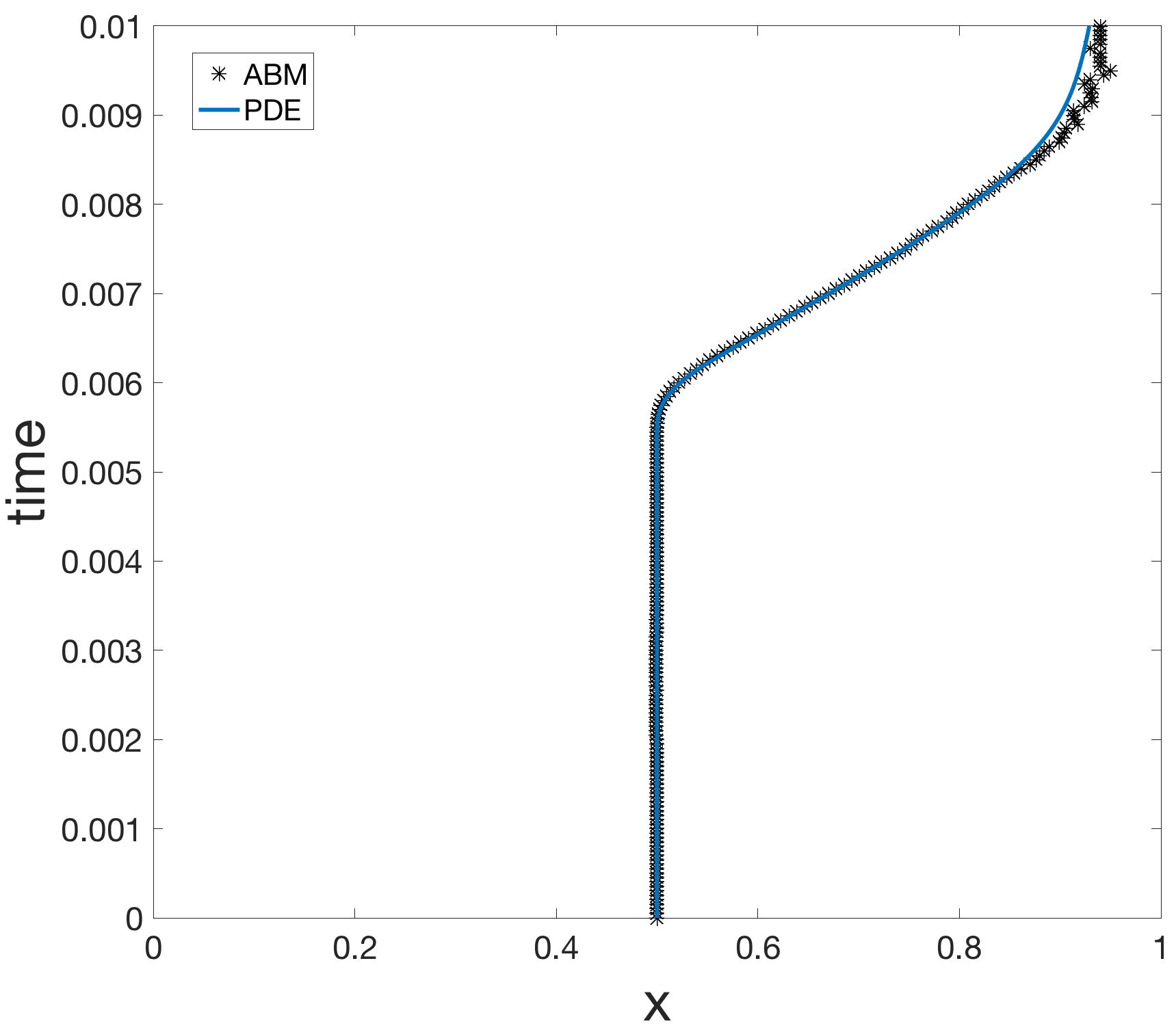}\label{fig: ExpMean}}
     \subfloat[][Standard Deviation]{
     \includegraphics[width=0.3\textwidth,height=1.5in]{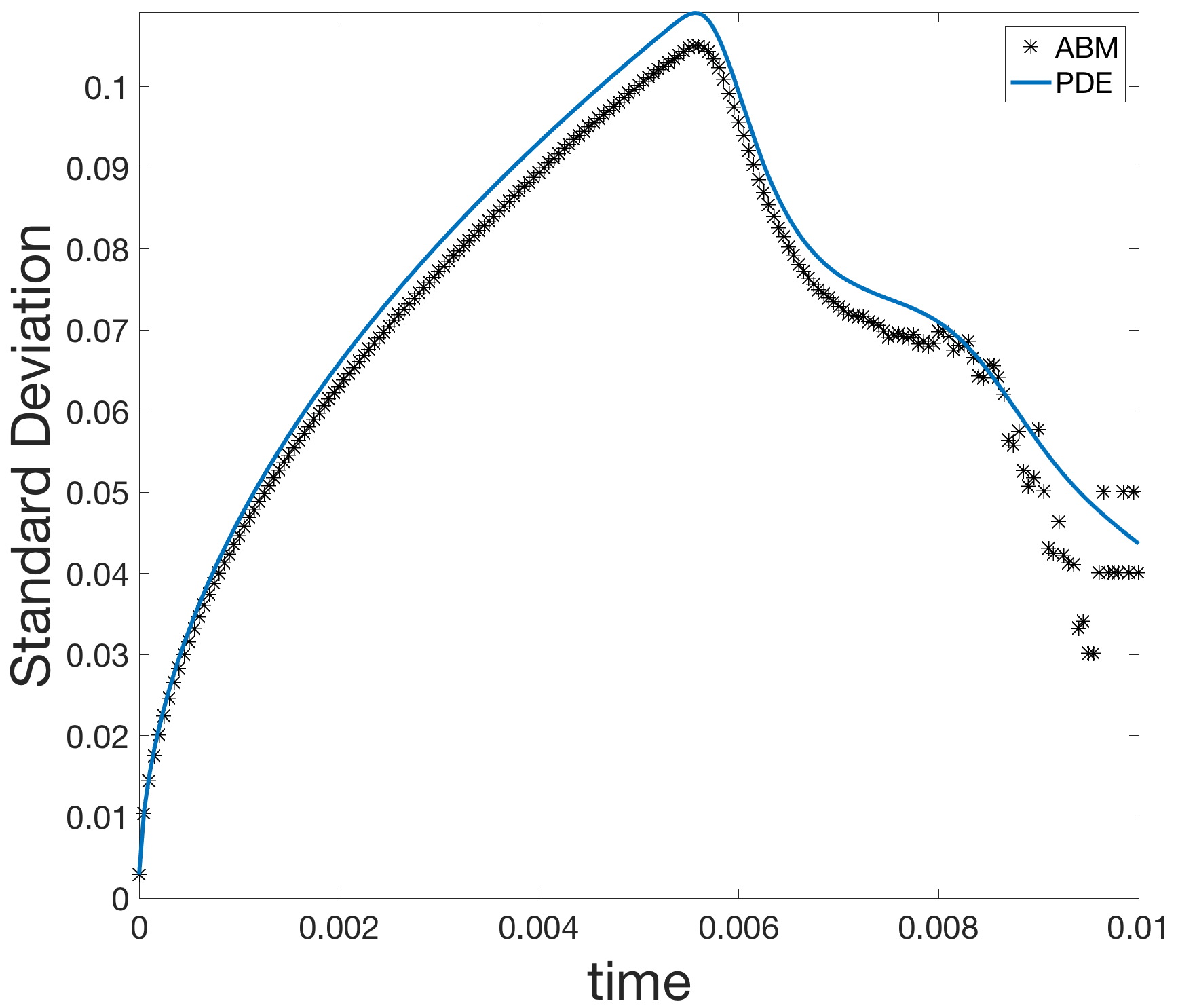}\label{fig: ExpStd}}
     \caption{Comparison of the survival probability, as well as mean and standard deviation of the live agent locations, for the ABM (black $*$) and the numerical PDE  solution (blue line) at each time-step \cso{for $C(x)=\exp(-x^2)$}.}
     \label{fig: ExpData}
\end{figure}

As shown in Fig.~\ref{fig: ExpStd}, the standard deviation of the agents locations is increasing for  $0\leq t \leq 0.005$, which corresponds to the time interval where most cells are alive (see survival probability in Fig.~\ref{fig: ExpSurvive}).  At $t=0.005$, agents with a cumulative absorption reaching $\xi_c$ begin to change state.  Cells to the right of $x_o=0.5$ tend to remain in the initialized state, which moves the mean to the right and reduces the variance.  A majority of the cells have changed state by $t=0.008$, where the cells that remain are those that continued to move right.  Thus, the standard deviation approaches zero.  Similar to Example 1, we see that as the number of agents in the ABM simulation approaches zero, the stochastic noise influences the variance (Fig.~\ref{fig: ExpStd}).

\subsubsection{Example 3: Biased random walk}\label{sec: biased}
\cmy{
Suppose the random walk has a constant bias, where $\ell$ and $r$ denote the probabilities of moving left or right, respectively.  Our absorption model is the following PDE  
\begin{equation}
\begin{cases}
& U_t + \beta(x)U_\xi = aU_x + DU_{xx}, \hspace{0.8cm} (x,\xi)\in \Omega^1, t>0\\
& U = \phi(x,\xi), \hspace{3.62cm} (x,\xi)\in \Omega^1, t=0 \\
& \lim_{|x|\to \infty}U = 0, \hspace{2.2cm}\qquad (x,\xi)\in \Omega^1, t>0,
\end{cases}
\label{eq: PDEBiased}
\end{equation}
where $a = \dx (\ell-r)/\dt$ and $D=\dx^2 (\ell + r)/(2\dt)$.
Note that the existence proof also holds if the agent moves with a constant bias.  After splitting the linear operator, the only difference between \eqref{eq: PDEBiased} and our initial cumulative absorption model equation \eqref{eq: PDE1d} is the form of the Green's function \[
G(x,t) = \frac{1}{4\pi D t}\exp\left\{-\frac{(x-at)^2}{4D t}\right\}, \quad a = \frac{\dx (\ell-r)}{\dt}.
\]
Replacing the diffusion Green's function with this advection-diffusion Green's function does not affect the existence proof in Section \ref{sec: Existence}.  Further, by integration by parts and using our far-field boundary condition, we can show that $\int_\Omega a UU_x\,d\xi\,dx = 0$.  Therefore, the biased model \eqref{eq: PDEBiased} satisfies Theorem 1, so we can prove that the PDE \eqref{eq: PDEBiased} is well-posed.

We set the chemical concentration as $C(x)=1/(1+10(x-0.5)^2)$ and the absorption capacitance as $\xi_c = 10\dx\dt$, the same as in Example 1.  However, in contrast to Example 1, we set the probability an agent moves left as $\ell=0.4$ and the probability an agent moves right as $r=0.6$.  We note that a larger density of agents tend to move to the right.  In fact, the graph in Fig.~\ref{fig: BiasMean} initially moves to the right in a straight line at a rate of $0.2\dx$.  This is due to the fact that only the biased movement determines the agent locations.  The graph of the mean location makes a sudden change around $t=0.005$, which is when some agents absorb above the threshold $\xi_c$ and begin leaving the live state.  At that time, the chemical profile begins to influence the mean location of agents, which also explains the shapes of the distributions in Fig.~\ref{fig: BiasSurf}.
\begin{figure}[H]
     \centering
     \subfloat[][ABM Simulation]{
     \includegraphics[width=0.4\textwidth]{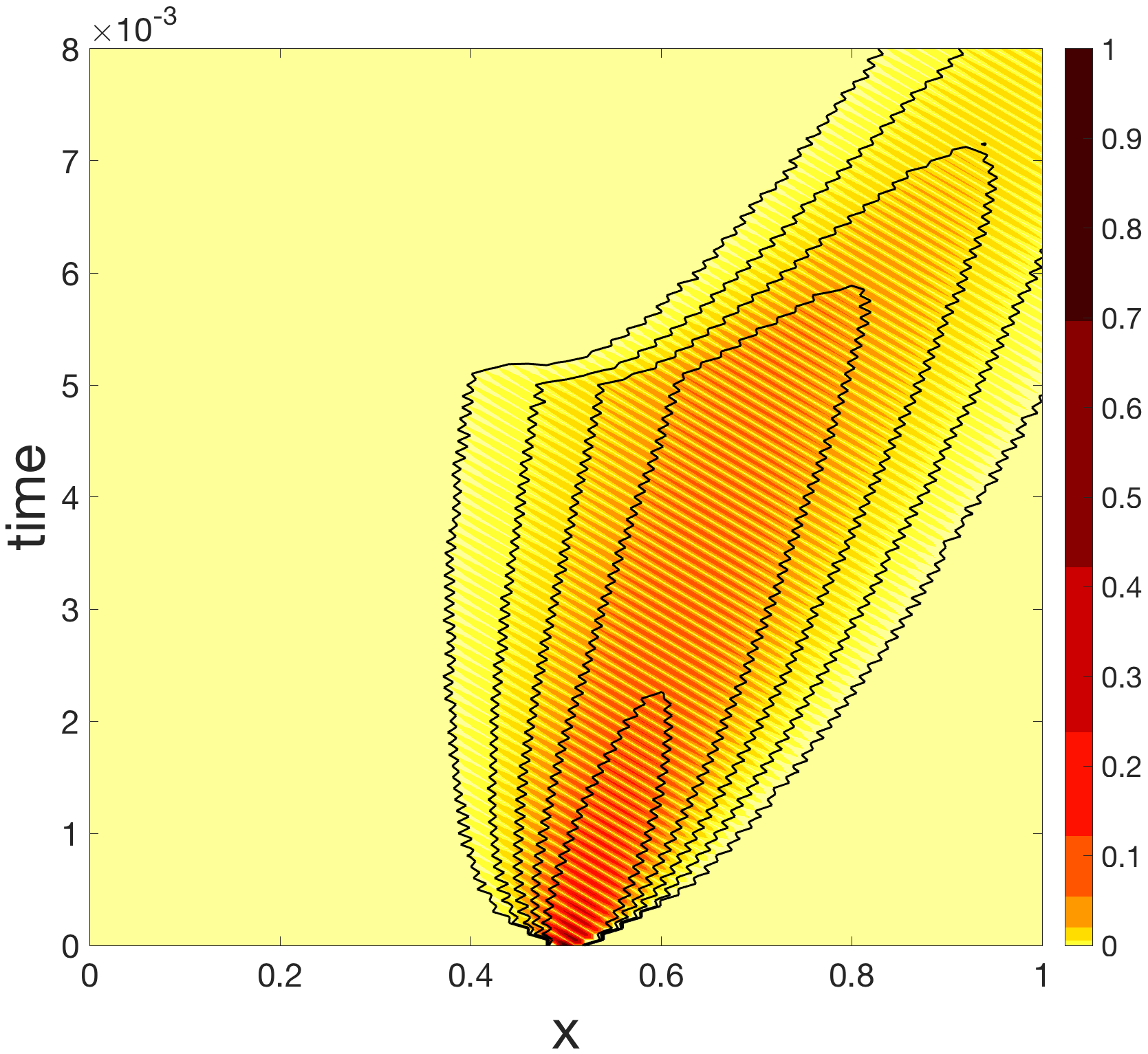}\label{fig: ABMBiasSurf}}
     \subfloat[][PDE Approximation]{
     \includegraphics[width=0.4\textwidth]{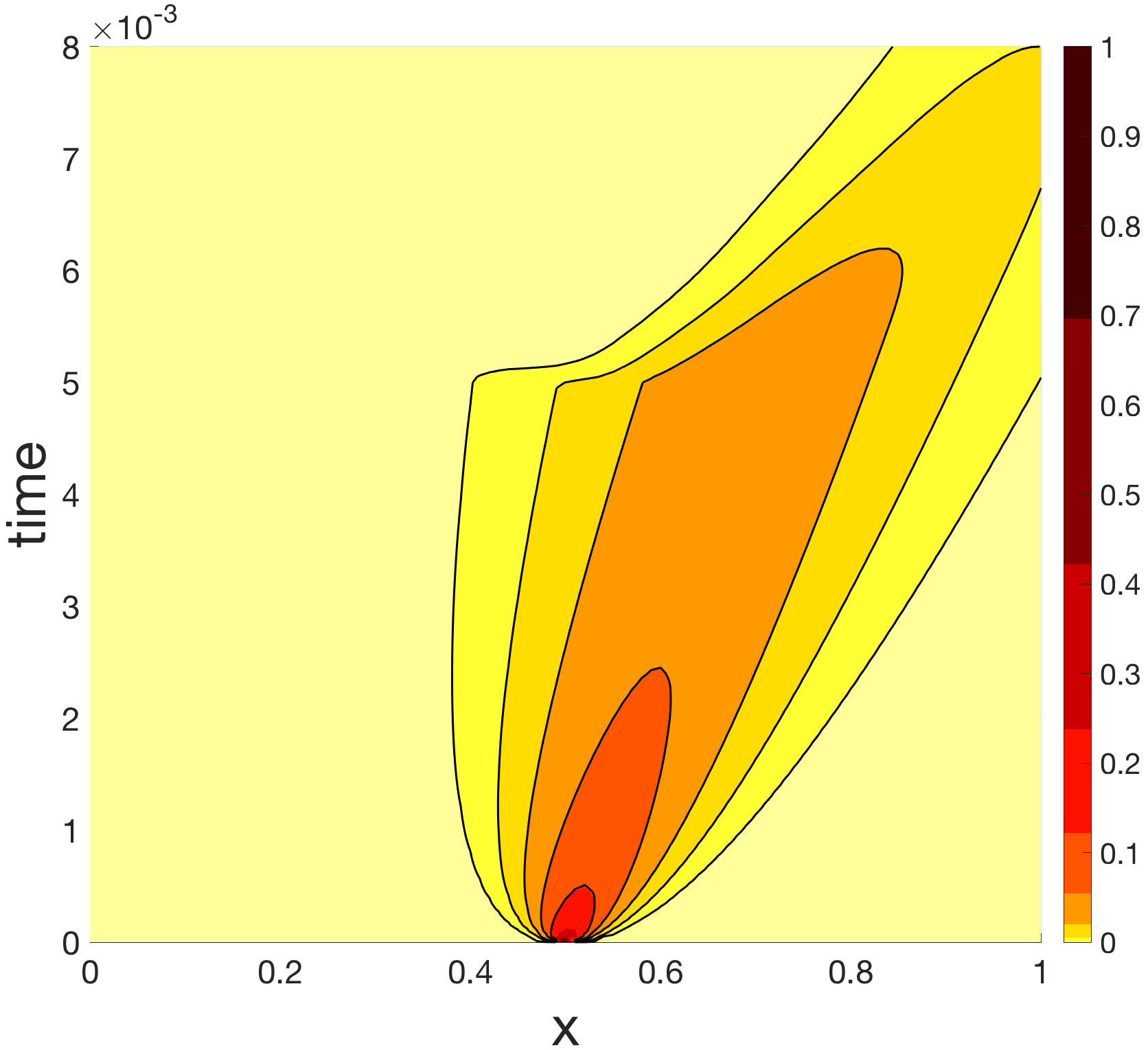}\label{fig: PDEBiasSurf}}
     \caption{Comparison of the probability distribution of live agents (shown in color) at locations $x\in[0,1]$ and at time points $t\in[0,0.1]$. The ABM results are a mean of 100,000 simulations. \cso{For both, the chemical concentration is $C(x)=1/(1+10(x-0.5)^2)$ and movement is biased to the right}.}
     \label{fig: BiasSurf}
\end{figure}

\begin{figure}[h]
     \centering
     \subfloat[][Survival Probability]{
     \includegraphics[width=0.3\textwidth]{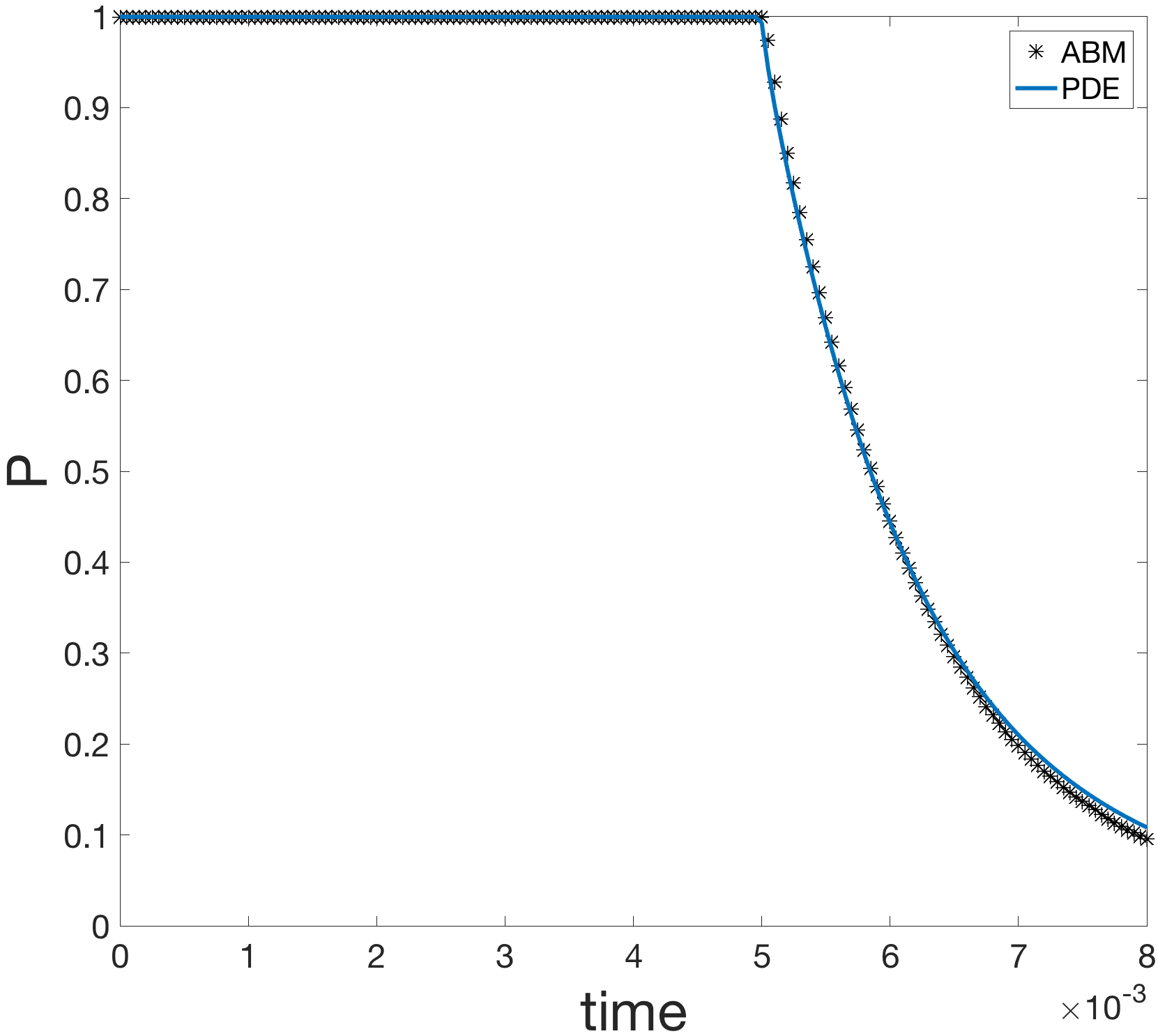}\label{fig: BiasSurvive}}
     \subfloat[][Mean Location]{
     \includegraphics[width=0.3\textwidth]{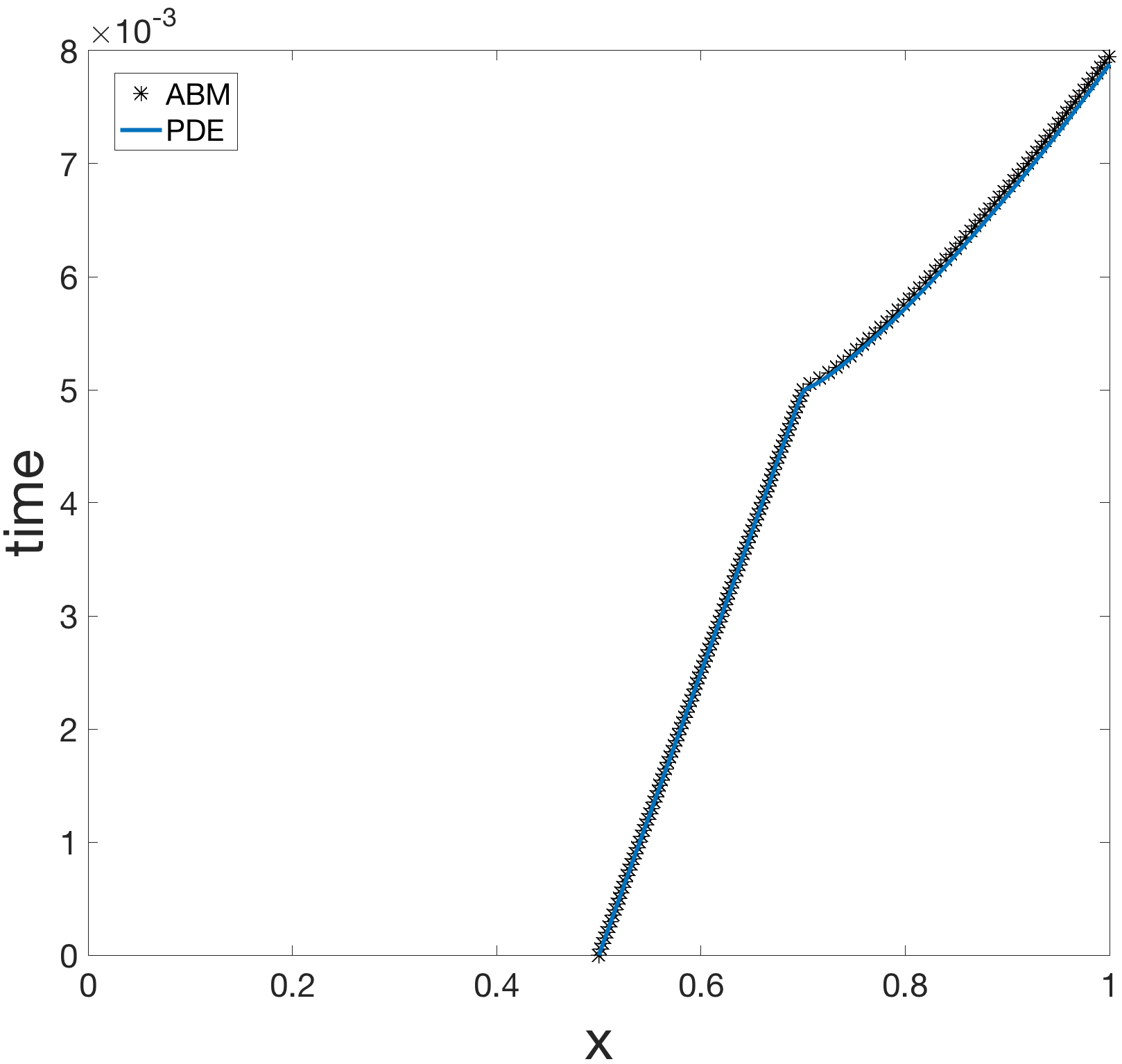}\label{fig: BiasMean}}
     \subfloat[][Standard Deviation]{
     \includegraphics[width=0.3\textwidth]{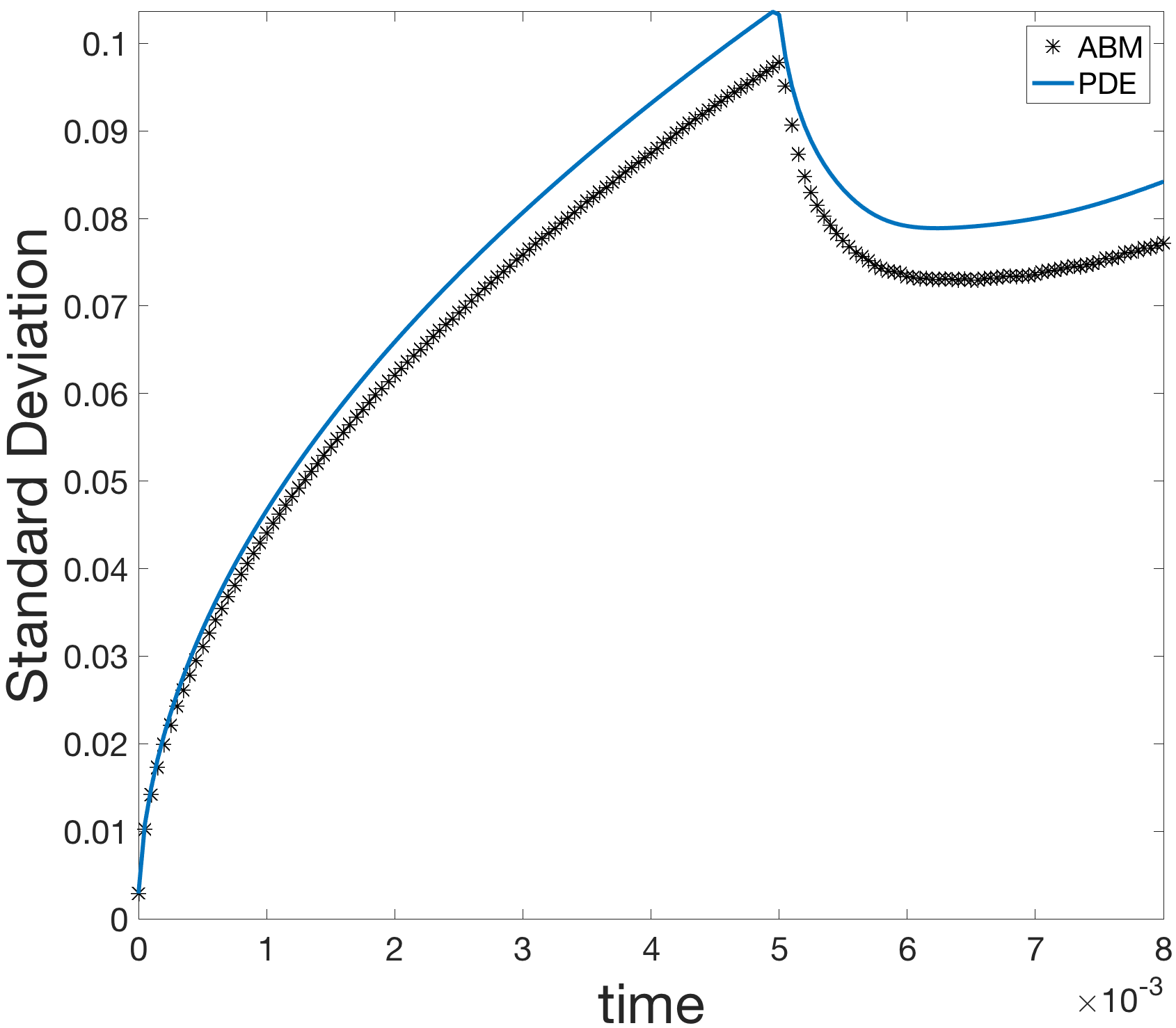}\label{fig: BiasStd}}
     \caption{Comparison of the survival probability, as well as mean and standard deviation of the live agent locations, for the ABM (black $*$) and the numerical PDE  solution (blue line) at each time-step for \cso{biased motion and $C(x)=1/(1+10(x-0.5)^2)$}. }
     \label{fig: BiasData}
\end{figure}
The survival probability for the ABM and continuum PDE model is shown in Fig.~\ref{fig: BiasSurvive}, and again there is good agreement between the ABM and PDE solution. In comparison to the unbiased movement case in Fig.~\ref{fig: RatlSurvive}, Fig.~\ref{fig: BiasSurvive} with biased movement begins decreasing at an earlier time, but then decreases at a slower rate.}

\subsection{The 2-dimensional model}
We can readily extend the analysis and numerical methods in Sections \ref{sec:analysis}-\ref{sec:numerics} to the 2-dimensional case.  
To account for the increased stochasticity of adding an additional dimension, we initialize 10 million agents. 
The agents in the ABM move with spatial step size of $\dx=\dy=0.01$ and time step $\dt=\dx^2/2$.  Similarly, the PDE model utilizes a spatial step size of $\cmy{dx = dy = \dx}$ and a time step of \cmy{$dt=\dt$}, and cumulative absorption of $d\xi=\xi_c/1000$. For both the ABM and PDE model, we set $\beta(\x)=\alpha \int_{B(\x,\dx/2)}C(\x)\,d\x$ where the chemical concentration is $C(x,y)=0.5(\sin(4\pi x)\sin(4\pi y)+1)$ and the chemical absorption threshold is $\xi_c=2\dx\dy\dt$.

The surface plot of the concentration local to the initialized agents in $[0,1]\times [0,1]$ is shown in \cmy{the dashed line contour plots in Fig.~\ref{fig: 2DABM} and Fig.~\ref{fig: 2DPDE}, where lighter colored lines denote a value closer to 0 and darker colored lines denote a value closer to 1.} 
The concentration is symmetric along the lines $y=x$ and $y=1-x$.
Near the initial location at $(0.5,0.5)$, there are local concentration minimums along the line $y=-x$.  
Thus, it makes sense that the probabilities for agents in the initial live state tend to be higher close to these chemical sinks, as shown in Figs.~\ref{fig: 2DABM} and \ref{fig: 2DPDE}.  
In fact, Figs.~\ref{fig: 2DABM}(b)-(c) and \ref{fig: 2DPDE}(b)-(c) show the probability density function mode bifurcation.  
That is, the chemical distribution causes $p_i^m$ to evolve into a bi-modal distribution, with each peak located on the line $y=1-x$ and equidistant to the line $y=x$. Again, when comparing the survival probability as a function of time, we observe excellent agreement between the ABM and continuum PDE (Fig.~\ref{fig: 2DSurvive}).
\begin{figure}[H]
\centering
\subfloat[][$t=0.0002$]{
\includegraphics[width=.22\textwidth]{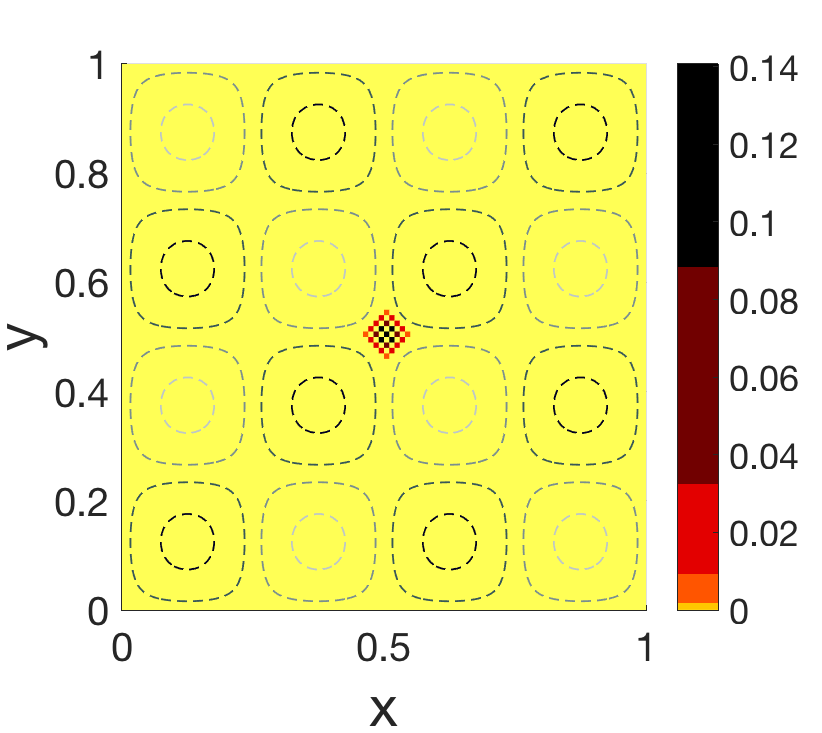}}
\,
\subfloat[][$t=0.0017$]{
\includegraphics[width=.22\textwidth]{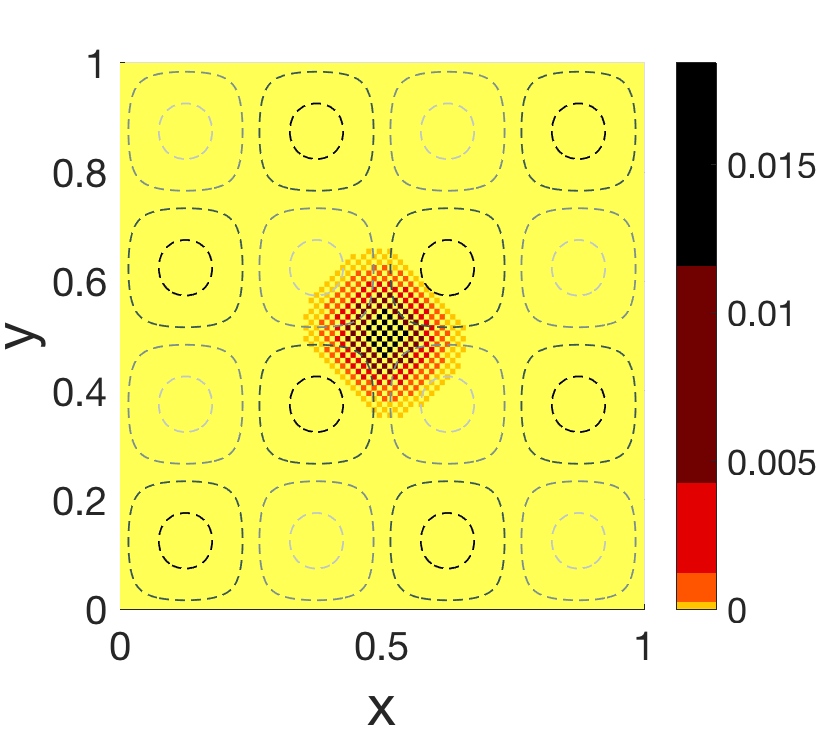}}
\,
\subfloat[][$t=0.0024$]{
\includegraphics[width=.22\textwidth]{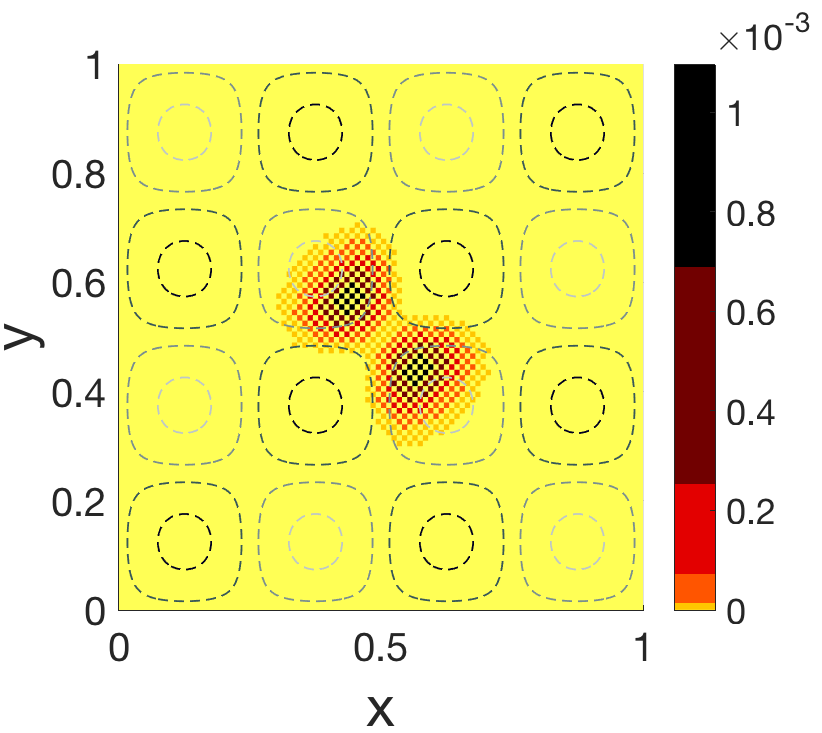}\label{fig: ABMsplit}}
\,
\subfloat[][$t=0.0050$]{
\includegraphics[width=.22\textwidth]{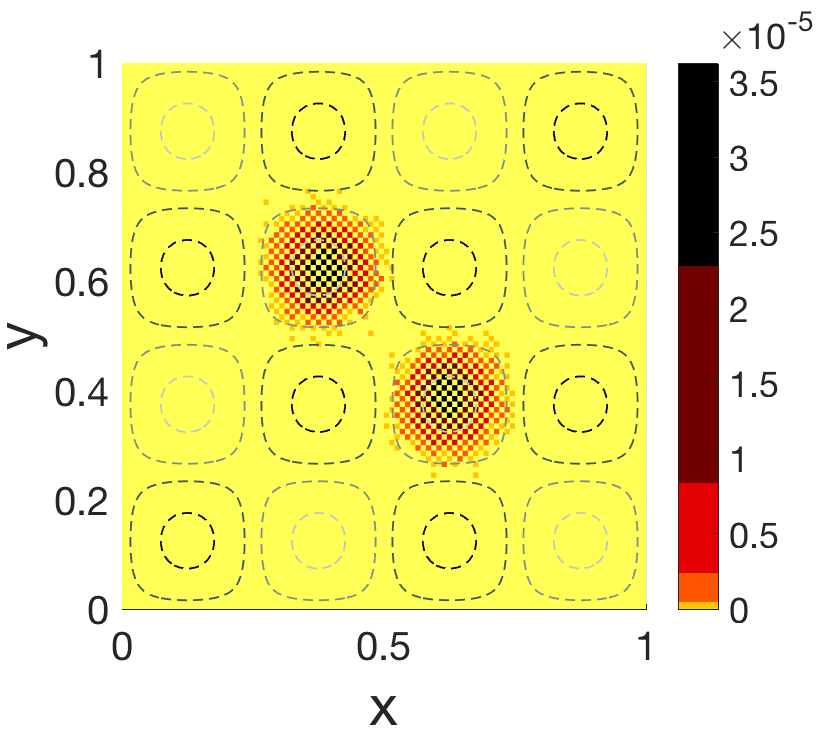}}
\caption{Probability distribution of live agents for the ABM (shown in color) in the region $[0,1]\times[0,1]$ at 4 different time points with $\alpha = 0.10$. The ABM results are a mean of 10 million agents.  \cmy{The dashed-line contour plot indicates the chemical concentration, $C(x,y)$.}}
\label{fig: 2DABM}
\end{figure}

\begin{figure}[H]
\centering
\subfloat[][$t=0.0002$]{
\includegraphics[width=.22\textwidth]{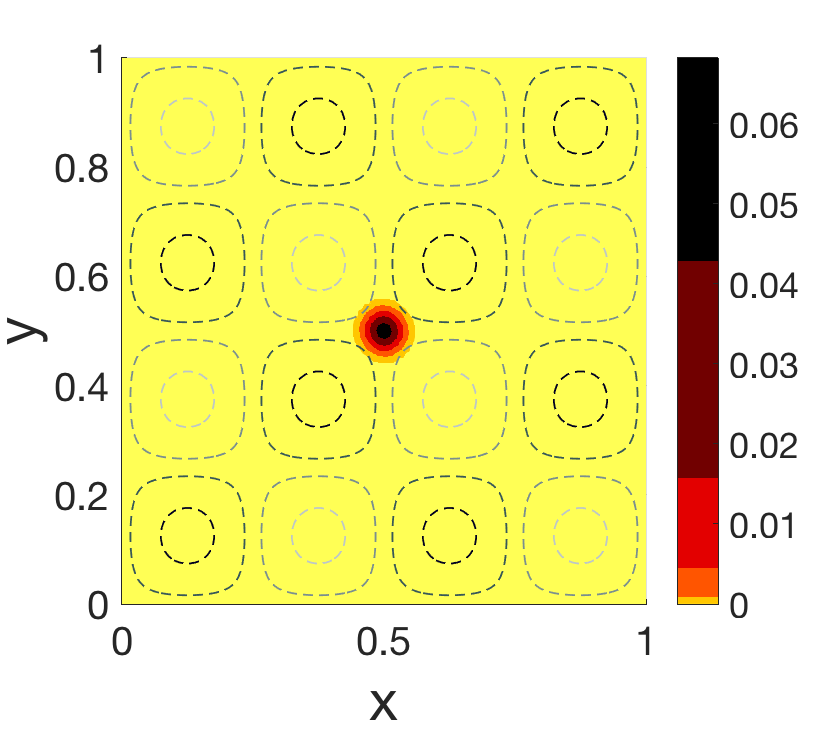}}
\,
\subfloat[][$t=0.0017$]{
\includegraphics[width=.22\textwidth]{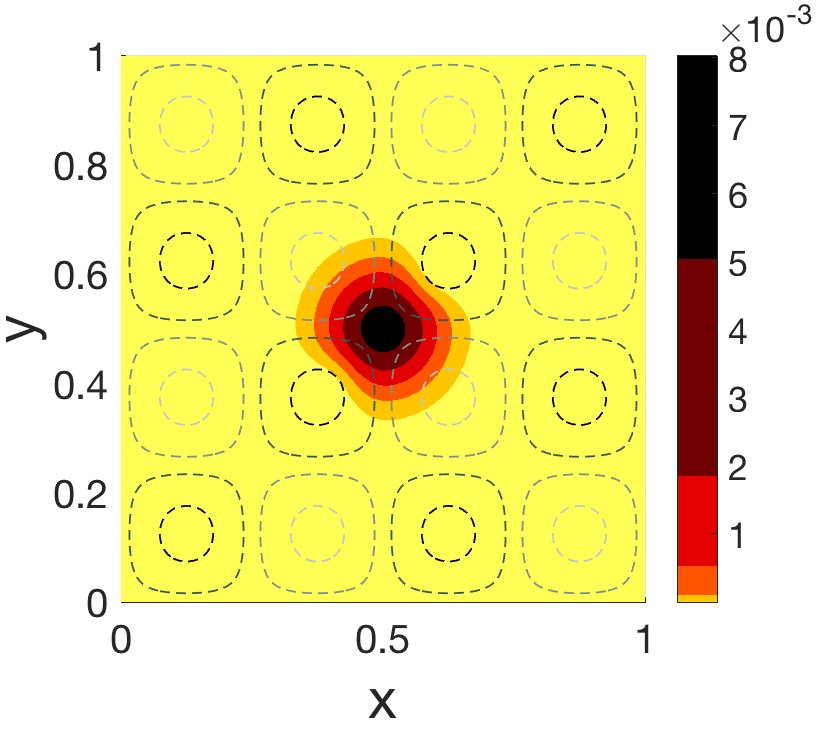}}
\,
\subfloat[][$t=0.0024$]{
\includegraphics[width=.22\textwidth]{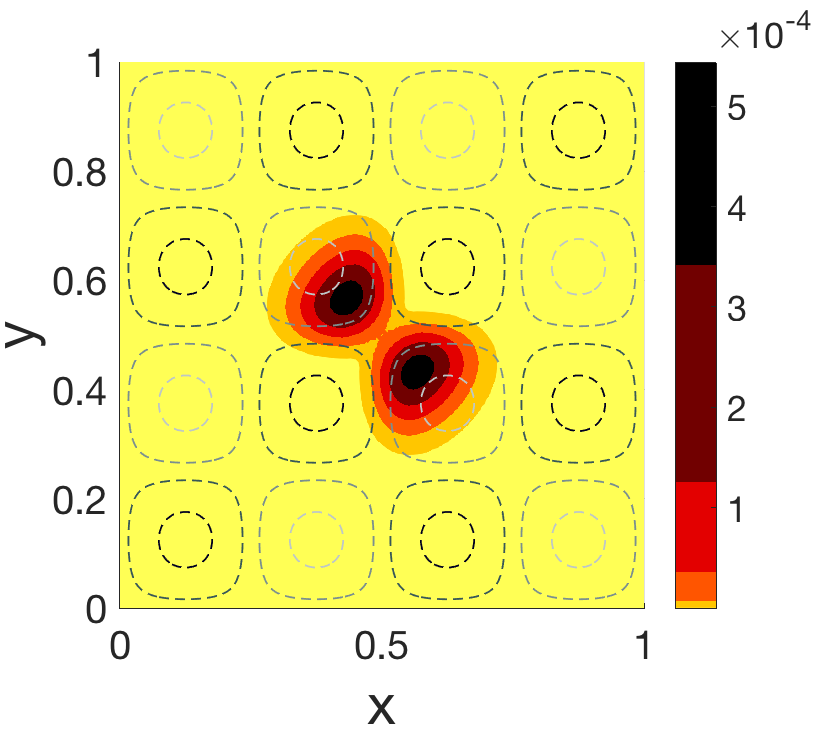}\label{fig: PDEsplit}}
\,
\subfloat[][$t=0.0050$]{
\includegraphics[width=.22\textwidth]{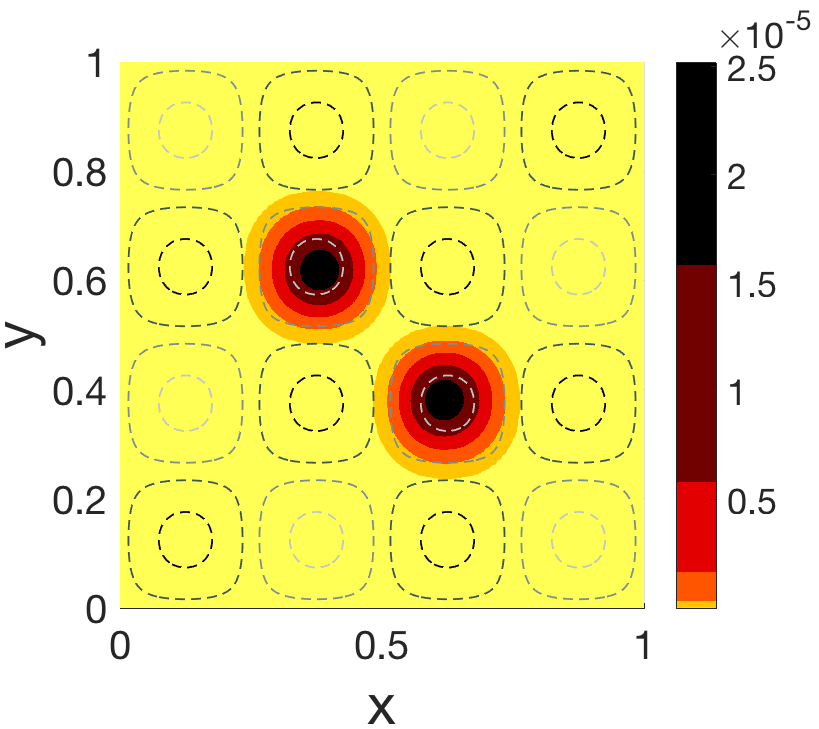}}
\caption{Probability distribution of live agents for the numerical PDE solution (shown in color) in the region $[0,1]\times[0,1]$ at 4 different time points with $\alpha = 0.10$.  \cmy{The dashed-line contour plot indicates the chemical concentration, $C(x,y)$}.}
\label{fig: 2DPDE}
\end{figure}

\begin{figure}[h]
     \centering
     \subfloat[][Survival Probability]{
     \includegraphics[width=0.3\textwidth]{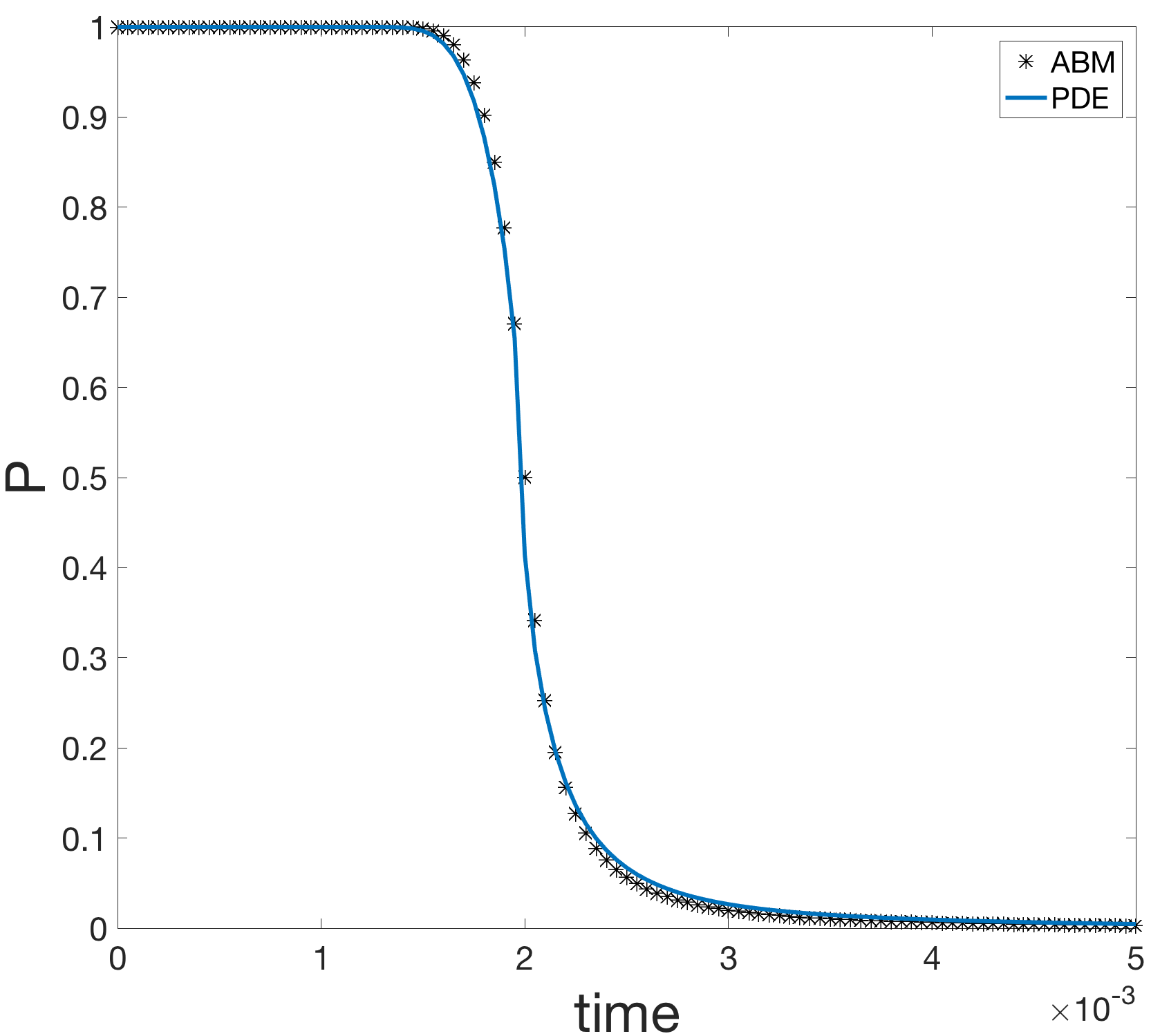}\label{fig: 2DSurvive}}
     \,
     \subfloat[][Mean Location]{
     \includegraphics[width=0.3\textwidth,height=1.45in]{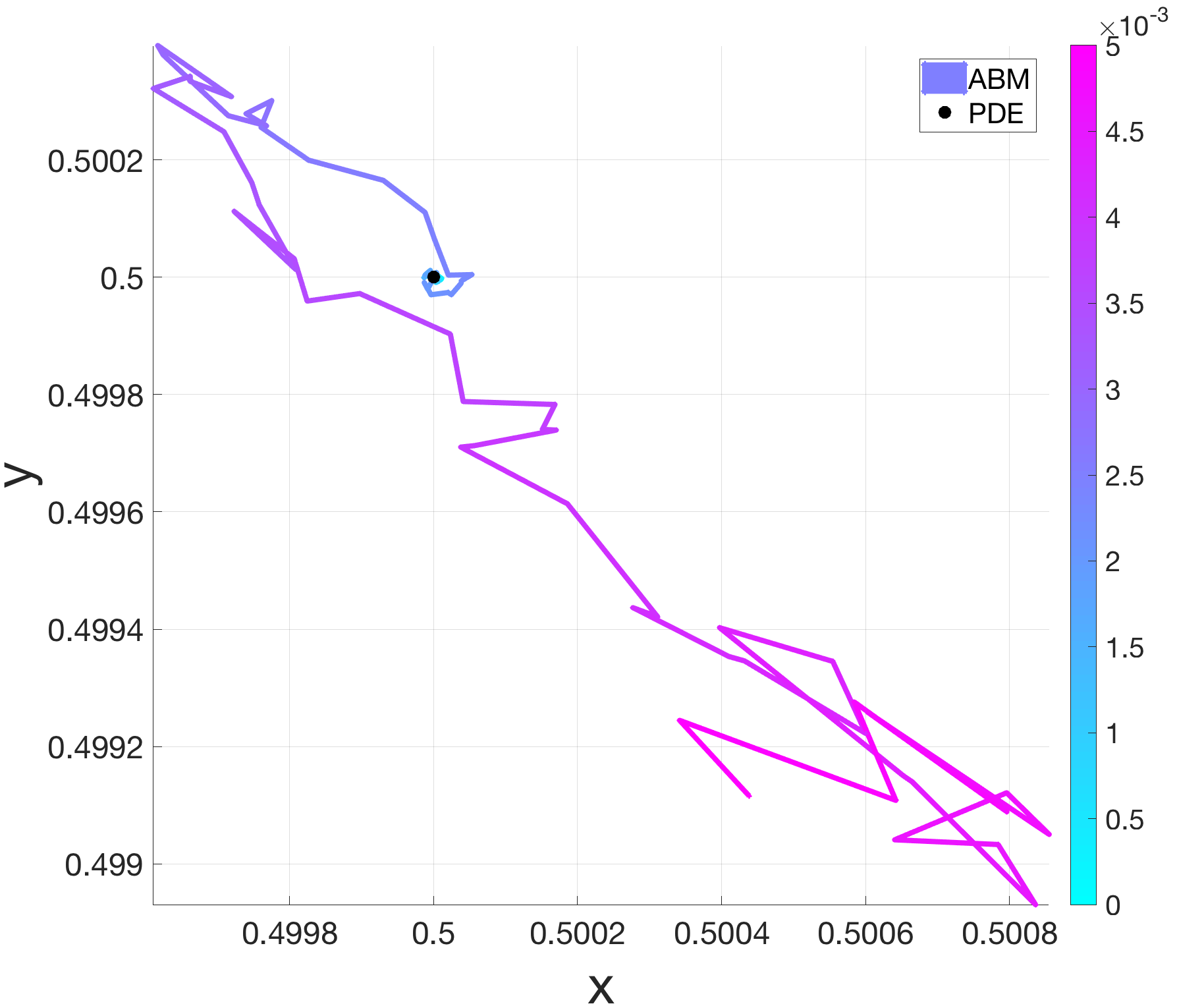}\label{fig: 2Dmean}}
     \subfloat[][Standard Deviation]{
     \includegraphics[width=0.3\textwidth]{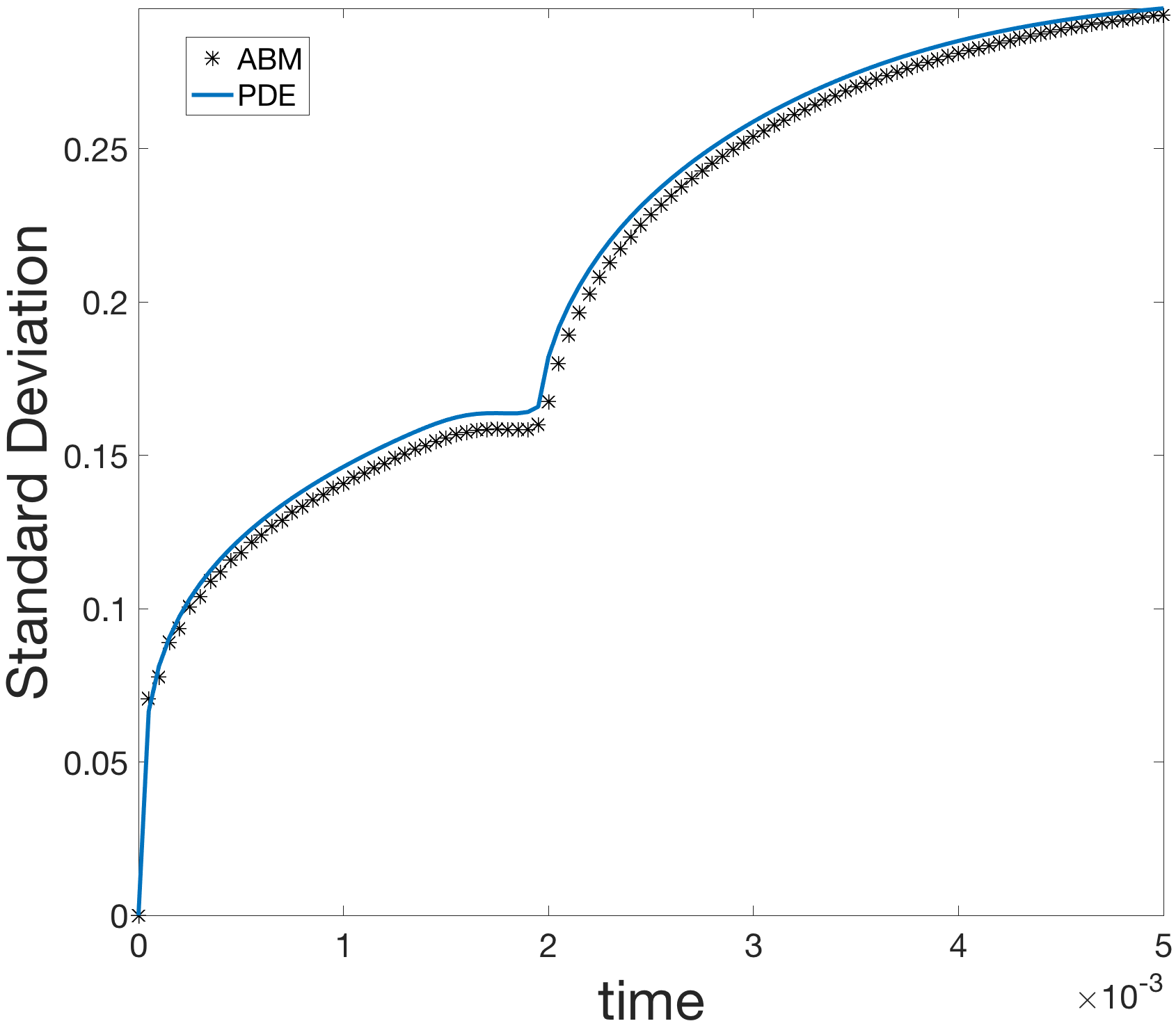}\label{fig: 2Dstd}}
     \caption{Comparison of the survival probability, mean, and standard deviation of the ABM and the semi-discrete numerical PDE solution at each time-step.  
The color of the ABM mean in (b) corresponds to the time-step.  Because the mean location of the numerical PDE approximation is located at (0.5,0.5) for every time-step, we label it using a black $\bullet$ to make a visual comparison with the ABM mean easier.
     }
     \label{fig: 2D}
\end{figure}

Fig.~\ref{fig: 2Dmean} demonstrates that the mean location of the PDE approximation remains constant at $(0.5,0.5)$.  The ABM mean is not constant.  However, since the ABM mean is contained within the region $B((0.5,0.5), \dx/2)$, and travels away from the PDE mean for times $t>0.003$, we can assume that this is due to the greater influence of stochastic noise as the number of agents in the initial state becomes relatively small.  
Since there are sufficiently many agents towards the end of the simulation and the mean during this simulation is within the control region $B((0.5,0.5),\dx/2)$, we see in Fig.~\ref{fig: 2Dstd} the standard deviation of the ABM data is not unduly influenced by the stochastic noise.  
Hence, the ABM and PDE standard deviation curves match reasonably well throughout the simulations.

\cmy{
We can see the difference in how the model develops if we decrease the absorption proportion parameter to $\alpha = 0.01$ in Fig.~\ref{fig: 2DABM2} and Fig.~\ref{fig: 2DPDE2}.  The agents diffuse for a longer time period before absorbing sufficient chemical to split.  With $\alpha = 0.01$, the distribution forms two peaks around $t=0.02$ (as shown in Fig.~\ref{fig: 2DABM2}(d), \ref{fig: 2DPDE2}(c)), whereas with $\alpha = 0.10$, the distribution forms two peaks around $t=0.0024$ (as shown in Fig.~\ref{fig: 2DABM}(c), \ref{fig: 2DPDE}(c)).  The pattern is different from $\alpha = 0.10$ as $t$ increases further in that since $p_i^m$ initially diffuses farther before changing states, the distribution will settle along additional chemical sinks (as shown in Fig.~\ref{fig: 2DABM2}(d), \ref{fig: 2DPDE2}(d)). 

\begin{figure}[H]
\centering
\subfloat[][$t=0.0002$]{
\includegraphics[width=.22\textwidth]{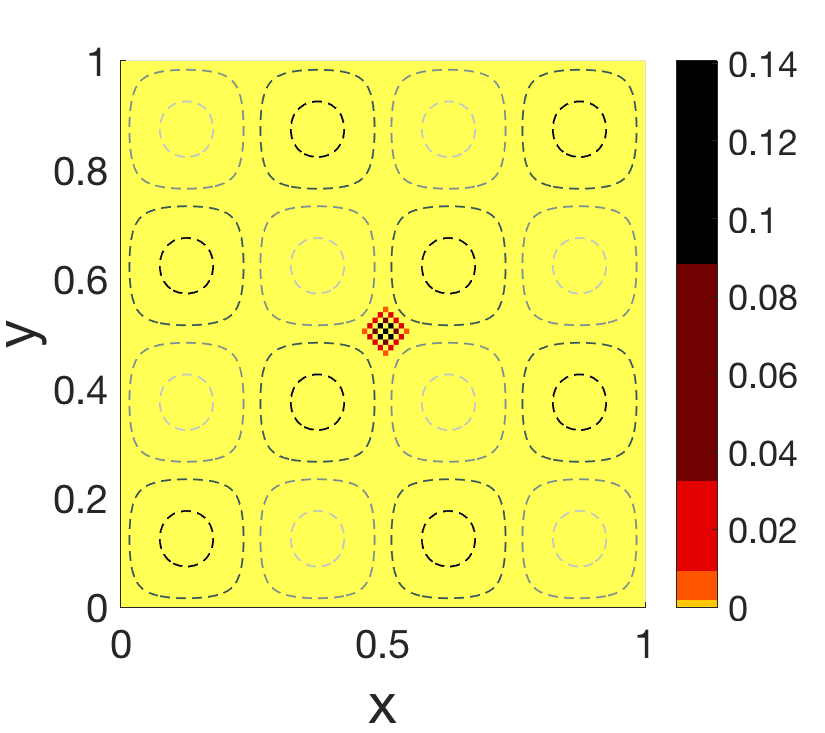}}
\,
\subfloat[][$t=0.0100$]{
\includegraphics[width=.22\textwidth]{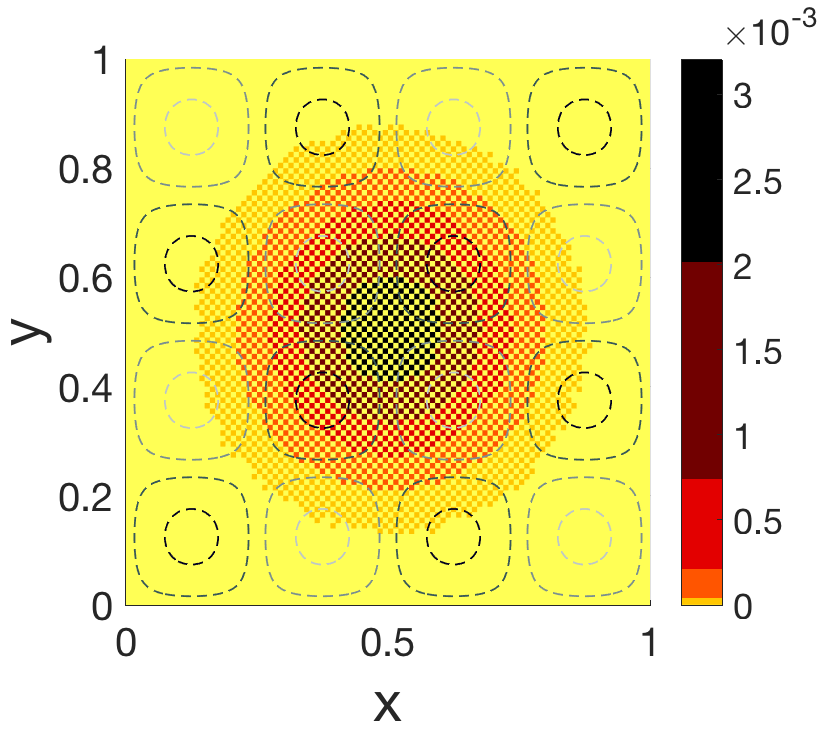}}
\,
\subfloat[][$t=0.0200$]{
\includegraphics[width=.22\textwidth]{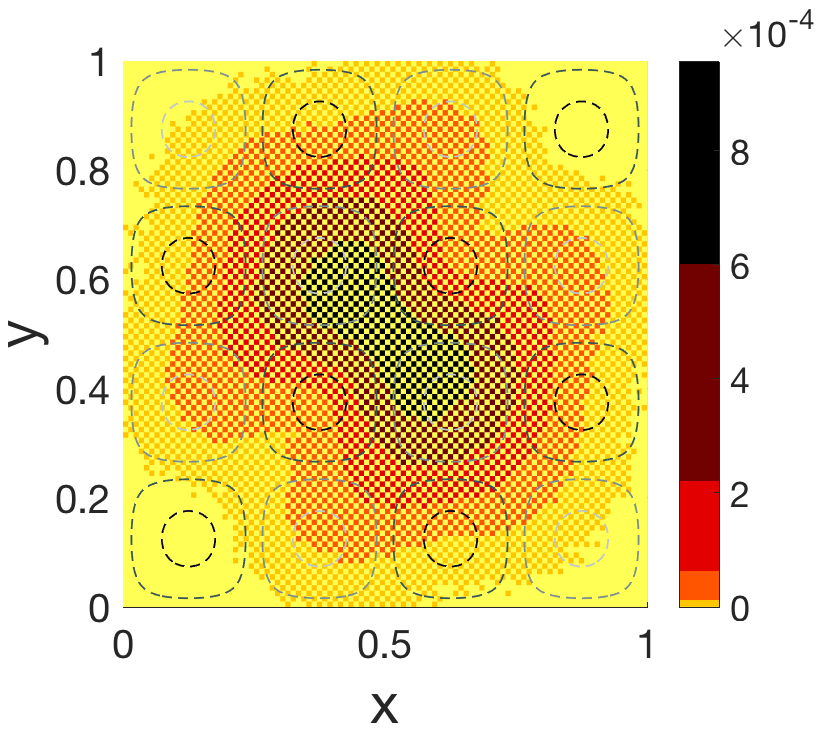}\label{fig: ABMsplit2}}
\,
\subfloat[][$t=0.0375$]{
\includegraphics[width=.22\textwidth]{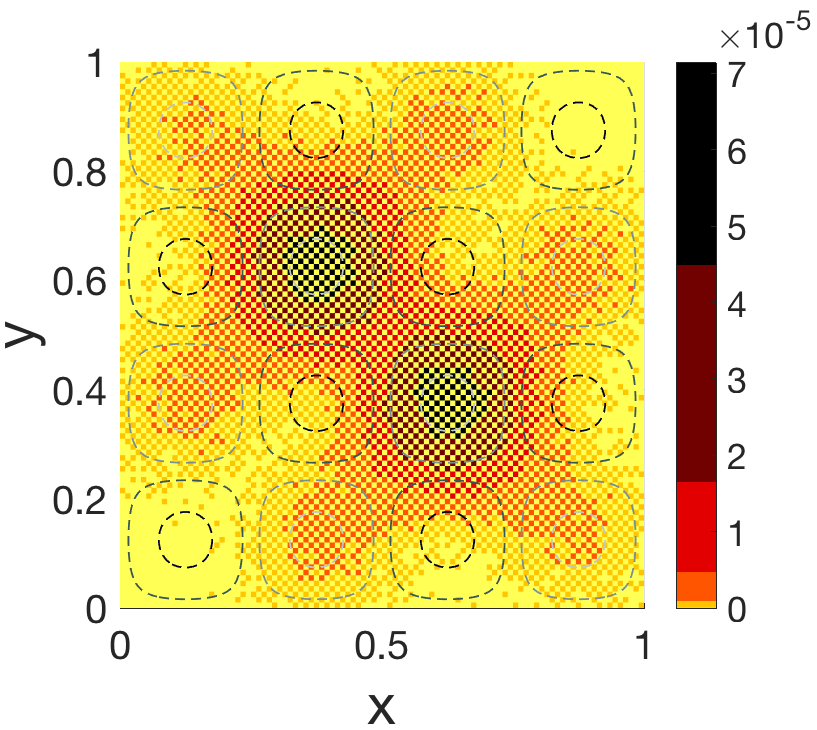}\label{fig: ABMsplit3}}
\caption{\cmy{Probability distribution of live agents for the ABM (shown in color) in the region $[0,1]\times[0,1]$ at 4 different time points with $\alpha = 0.01$. The ABM results are a mean of 10 million agents.  The dashed-line contour plot indicates the chemical concentration, $C(x,y)$.}}
\label{fig: 2DABM2}
\end{figure}

\begin{figure}[H]
\centering
\subfloat[][$t=0.0002$]{
\includegraphics[width=.22\textwidth]{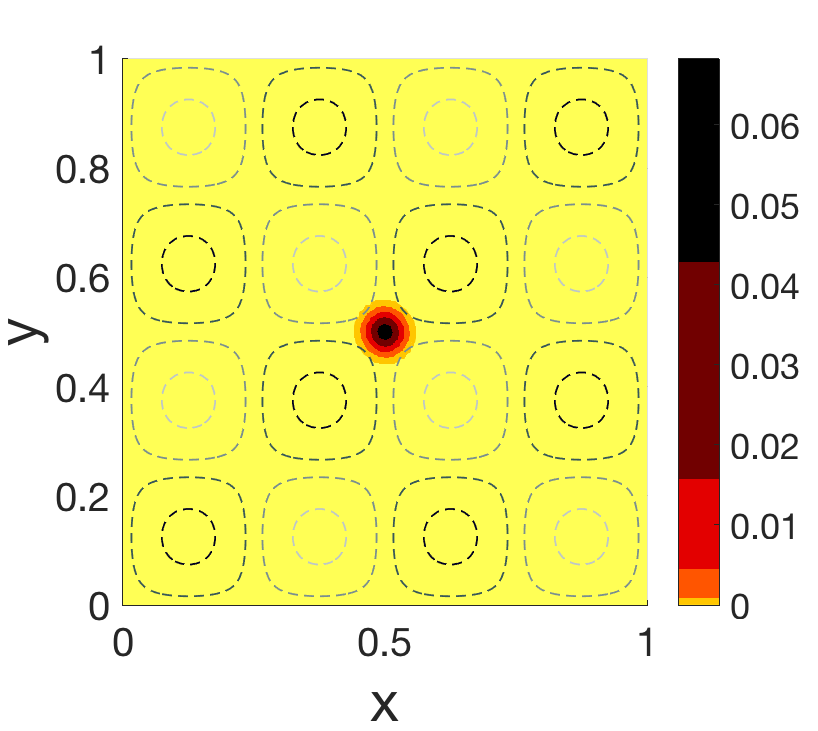}}
\,
\subfloat[][$t=0.0100$]{
\includegraphics[width=.22\textwidth]{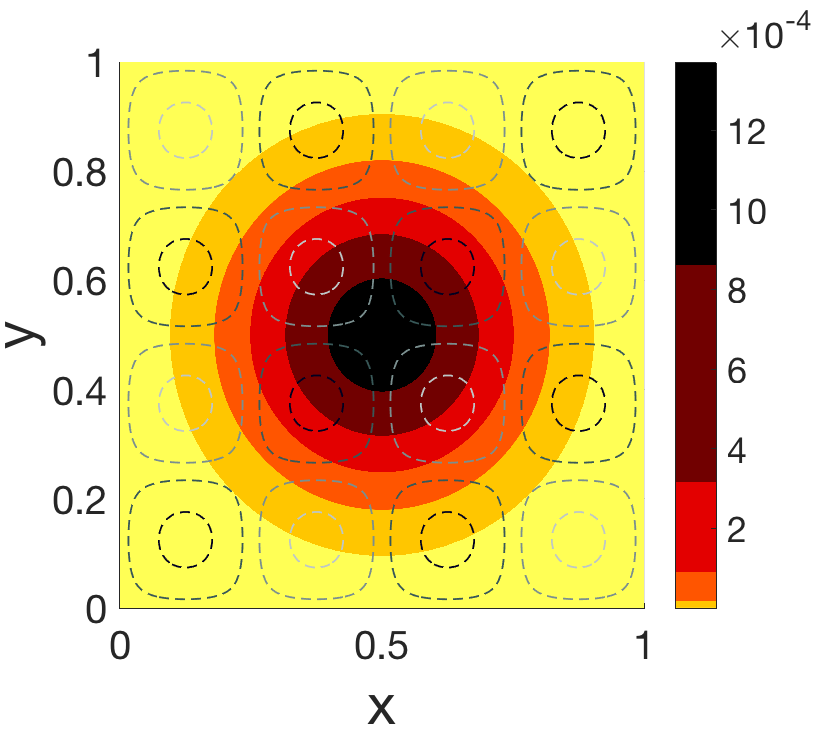}}
\,
\subfloat[][$t=0.0200$]{
\includegraphics[width=.22\textwidth]{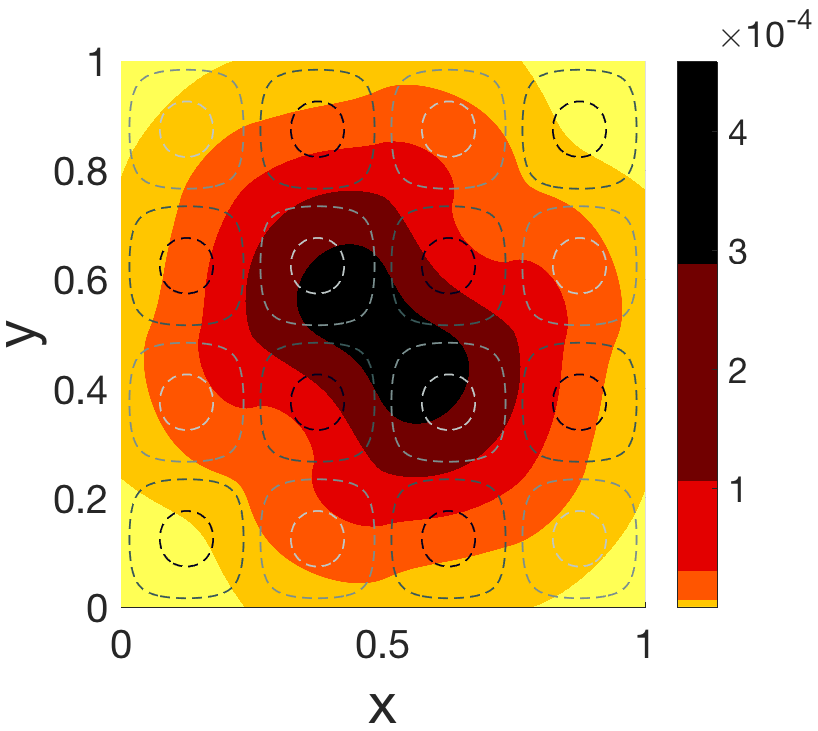}\label{fig: PDEsplit2}}
\,
\subfloat[][$t=0.0375$]{
\includegraphics[width=.22\textwidth]{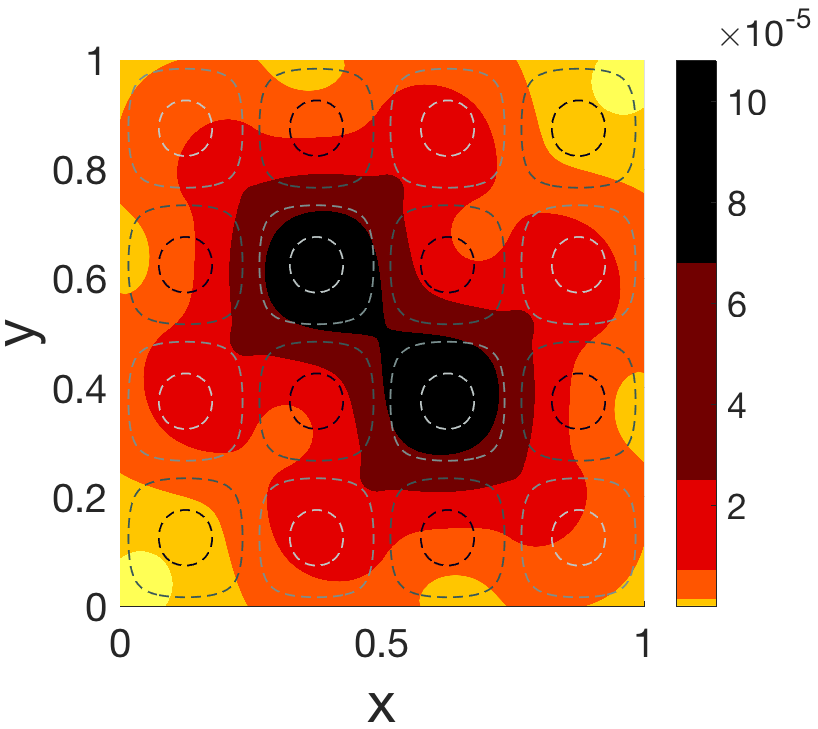}\label{fig: PDEsplit3}}
\caption{Probability distribution of live agents for the numerical PDE solution (shown in color) in the region $[0,1]\times[0,1]$ at 4 different time points with $\alpha = 0.01$.  The dashed-line contour plot indicates the chemical concentration, $C(x,y)$.}
\label{fig: 2DPDE2}
\end{figure}
If we instead choose an $\alpha \ll 1$, the diffusion time of the agents is much faster than the time for the agents to absorb chemical to capacitance.  In that case, we can simplify and re-frame the model as a diffusion-dominant absorption model.  The derivation and examples of such a model are further developed in Section \ref{sec: scaling}.}

\cmy{
\section{Relative scaling approximations \label{sec: scaling}}
The above examples assumed a particular scaling of parameters in \eqref{eq: PDE1d} and \eqref{eq: PDE}.  However, other scaling relations may be possible and we can further simplify the equations or analysis.
First, let us non-dimensionalize the absorption variable, $\xi$, by using the scaling factor $\xi_c$ to obtain $\bar{\xi} = \xi/\xi_c$.} \cso{We can rewrite \eqref{eq: PDE1d} as}
\cmy{
\begin{equation}
\begin{cases}
& U_t + \frac{\beta(x)}{\xi_c}U_{\bar{\xi}} = DU_{xx}, \hspace{1.65cm} (x,\bar{\xi})\in \Omega^1, t>0\\
& U = \phi(x,\bar{\xi}), \hspace{3.15cm} (x,\bar{\xi})\in \Omega^1, t=0 \\
& \lim_{|x|\to \infty}U = 0, \hspace{1.7cm}\qquad (x,\bar{\xi})\in \Omega^1, t>0,
\end{cases}
\label{eq: PDE1dNonD}
\end{equation}
with the agent changing state when $\xi \geq 1$.} \cso{There are three different regimes based on the term $\beta(x)/(D\xi_c)$. } 

\cmy{First, we will investigate when $\beta(x)/(D\xi_c) \gg 1$.  This occurs when the agent absorbs chemical above its capacitance before diffusion moves the agent, as can be observed in Fig \ref{fig: AbsorptionDominant}(a).  Assuming $\beta(x)>0$ is continuous, we can asymptotically simplify \eqref{eq: PDE1dNonD} to
\begin{equation}
\begin{cases}
& U_t + \frac{\beta(x)}{\xi_c}U_{\bar{\xi}} = 0, \hspace{1.5cm} (x,\bar{\xi})\in \Omega^1, t>0\\
& U = \phi(x,\bar{\xi}), \hspace{2.2cm} (x,\bar{\xi})\in \Omega^1, t=0,
\end{cases}
\label{eq: PDE1dNonDAbs}
\end{equation}
with a solution $U(x,t,\bar{\xi}) = \phi(x,\bar{\xi}-\beta(x)t/\xi_c)$ for all $x\in \R$.  An example solution of the total density in the live state in an absorption-dominant parameter regime is shown in Fig \ref{fig: AbsorptionDominant}(b).  Note that as $\langle \beta(x) \rangle / (D\xi_c)$ approaches $\infty$, the ABM and PDE densities converge.

\begin{figure}[h]
\centering
\subfloat[][ABM simulation surface plot.]{
\includegraphics[width=.4\textwidth]{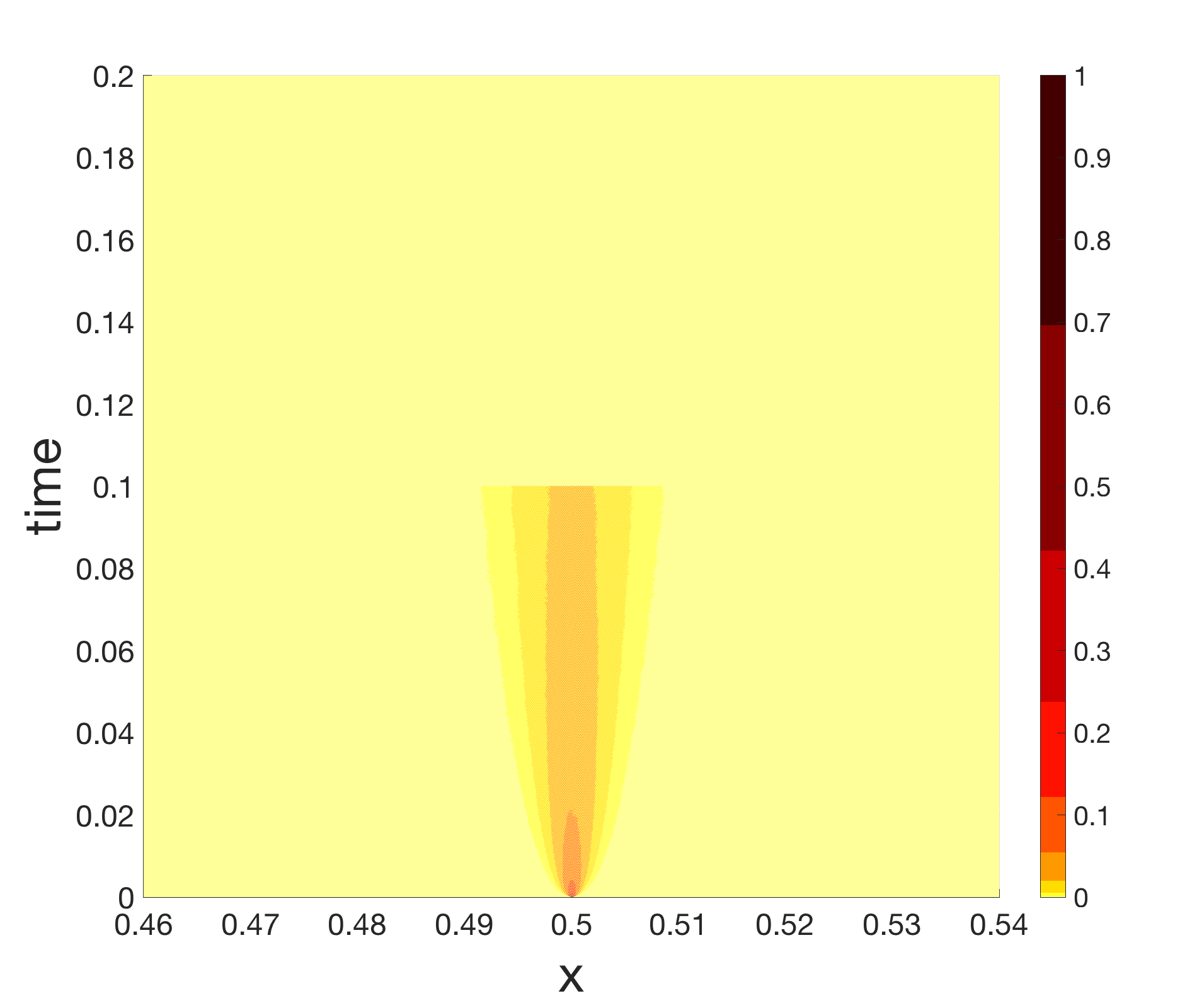}
\label{fig: AbsorptionDominant1}}
\,
\subfloat[][Comparison between ABM and PDE density in live state with respect to time.]{
\includegraphics[width=.4\textwidth]{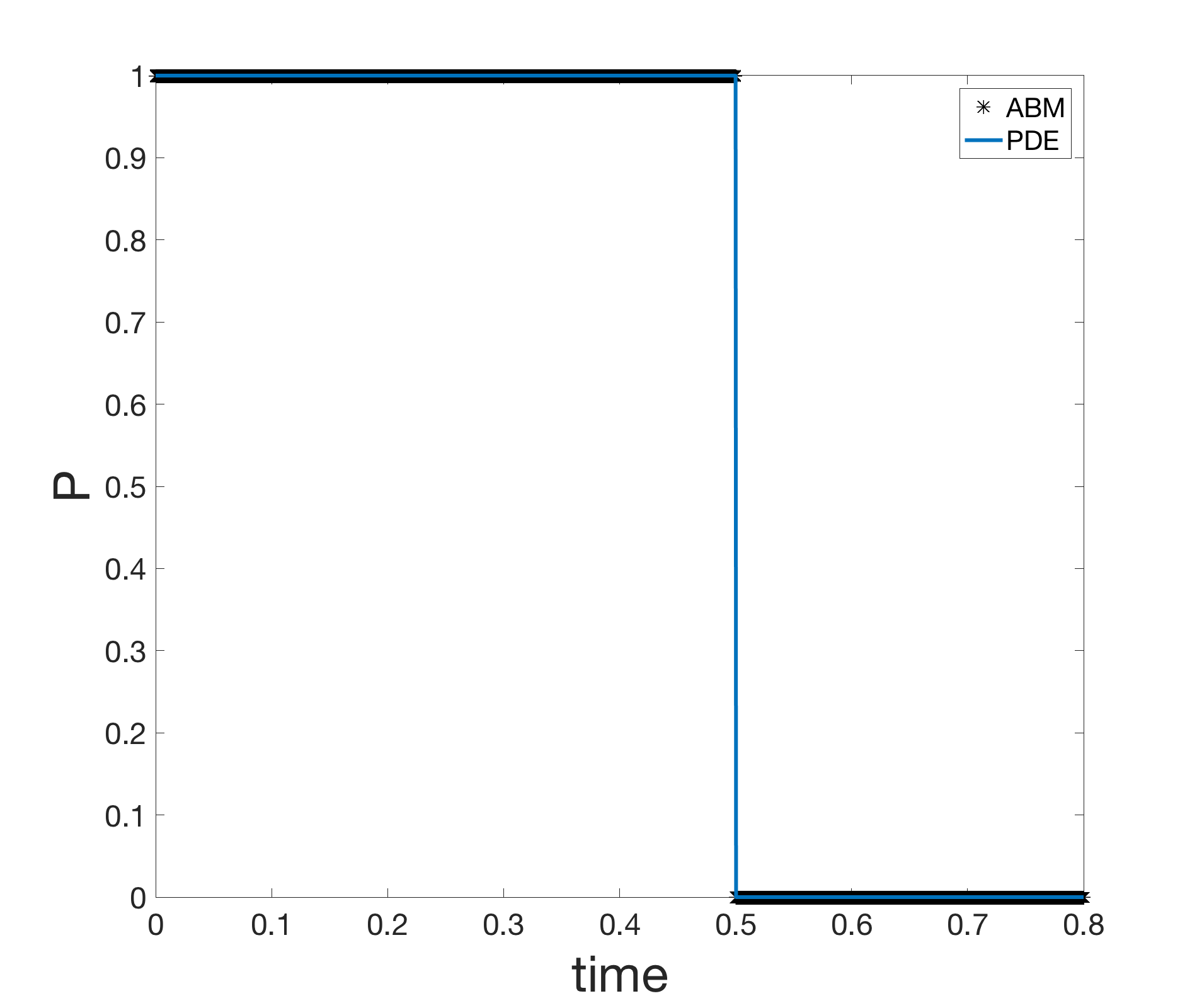}\label{fig: AbsorptionDominant2}}
\caption{Absorption-dominant parameter regime comparison between ABM and PDE density in live state with respect to time.  
Here $\dx=\dt=1\times 10^{-4}$, $\beta(x)=\dx 0.1 \frac{1}{2}(1+\sin \pi x)$, $\xi_c=1\times 10^{-6}$.  Agents are initialized at $x_0 = 0.5$.  
Thus, \cso{$\beta(x_0)/(D \xi_c ) = 200,000$.}}
\label{fig: AbsorptionDominant}
\end{figure}

Second, we will investigate when $\beta(x)/(D\xi_c) \ll 1$.  This occurs when the agent diffuses much faster than the agent absorbs chemical, as is shown in Fig \ref{fig: DiffusionDominant}(a).  
We cannot make an asymptotic argument, ignoring the absorption term, since that is our primary interest in this model.  Thus, we make the assumption that the density of the agents in the live state quickly becomes uniform over the spatial coordinate.  This assumption allows us to collapse the spatial coordinate and have a solution solely in absorption and time dimensions.  
We can define the PDE initial condition as $\bar{\phi}(\bar{\xi}) := \int_\R \phi(x,\bar{\xi})\,dx$ and we can redefine the PDE absorption term as the average value of $\beta$ in the spatial domain, $ \langle \beta(x) \rangle$, to obtain the following PDE,
\begin{equation}
\begin{cases}
& U_t + \frac{\langle \beta(x) \rangle}{\xi_c}U_{\bar{\xi}} = 0, \hspace{1.5cm} \bar{\xi} \in [0,\infty), t>0\\
& U = \bar{\phi}(\bar{\xi}), \hspace{3.15cm} \bar{\xi}\in [0,\infty), t=0.
\end{cases}
\label{eq: PDE1dNonDDiff}
\end{equation}
The solution is $U(\bar{\xi},t) = \bar{\phi}\left(\bar{\xi}-\langle \beta(x) \rangle t/\xi_c\right)$.
An example solution of the total density in the live state in a diffusion-dominant parameter regime is shown in Fig \ref{fig: DiffusionDominant}(b).  Note that as $\langle \beta(x) \rangle / (D\xi_c)$ approaches 0, the ABM and PDE densities converge.

\begin{figure}[h]
\centering
\subfloat[][ABM simulation surface plot.]{
\includegraphics[width=.4\textwidth]{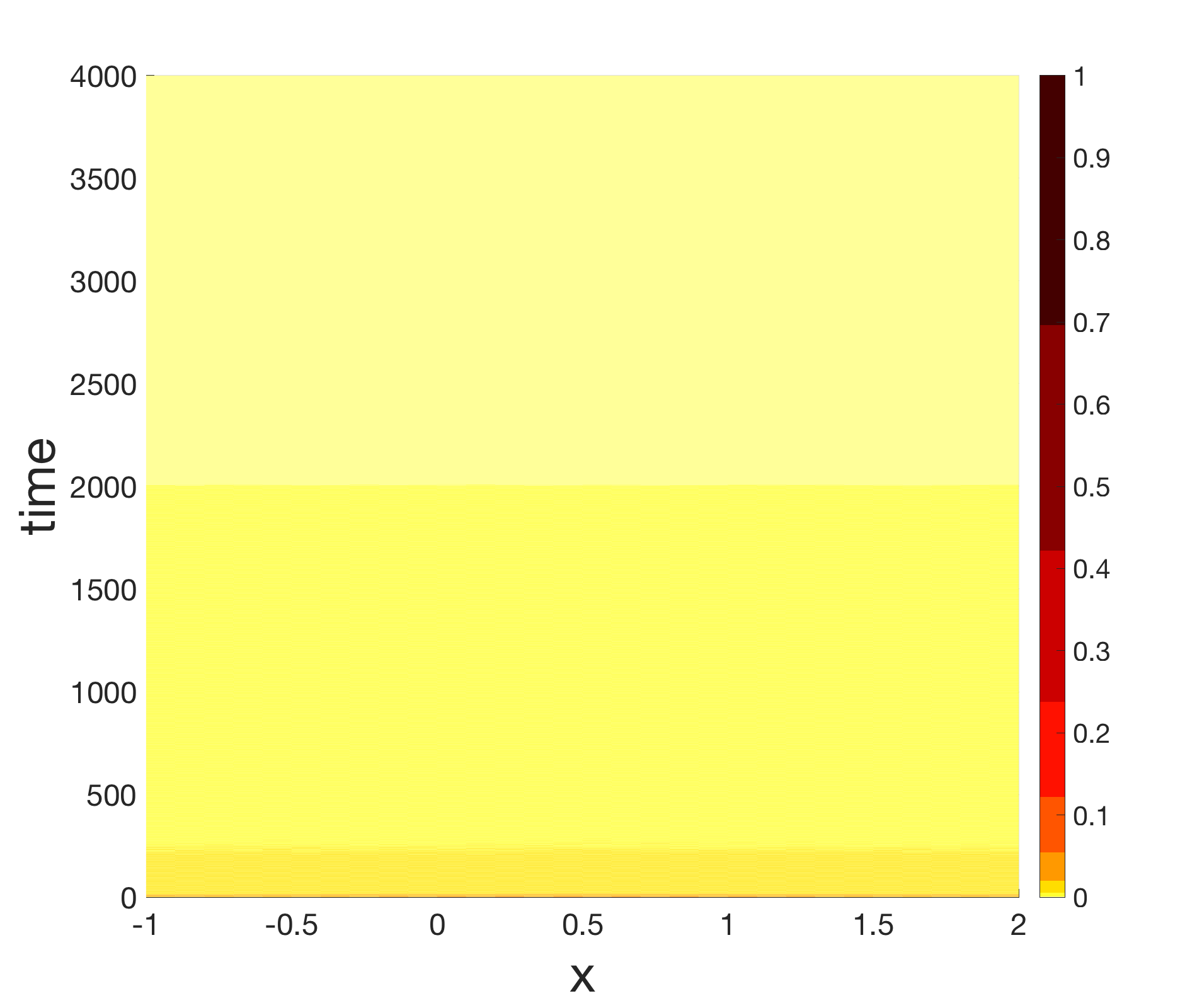}\label{fig: DiffusionDominant1}}
\,
\subfloat[][Comparison between ABM and PDE density in live state with respect to time.]{
\includegraphics[width=.4\textwidth]{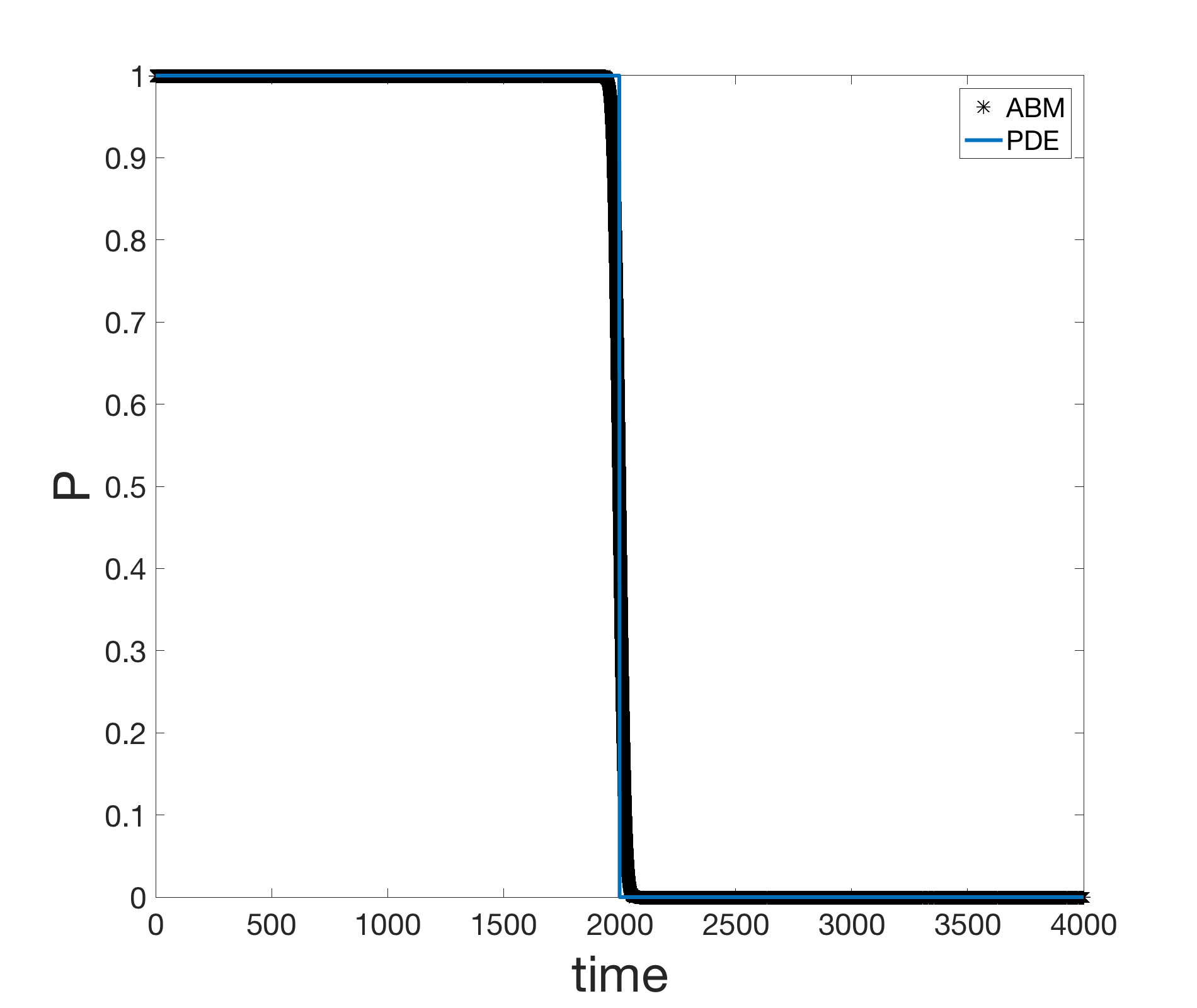}\label{fig: DiffusionDominant2}}
\caption{Diffusion-dominant parameter regime comparison between ABM and PDE density in live state with respect to time.  
Here $\dx=0.1$, $\dt=0.01$, $\beta(x)=\dx 0.1 \frac{1}{2}(1+\sin \pi x)$, $\xi_c=1$. 
Agents are initialized at $x_0 = 0.5$.  
Thus, $\langle \beta(x) \rangle /(D\xi_c) = 0.001$.
}
\label{fig: DiffusionDominant}
\end{figure}
}

\section{Discussion and Conclusions}
In this work, we have developed a continuum PDE approximation to a stochastic agent-based model that has a cumulative coupling to the environment. We have shown through simulations that the ABM agrees qualitatively with the governing PDE. We have analyzed the newly developed PDE, showing that we have developed a stable, well-posed equation. 

\cso{The modeling framework developed assumes that the cells have cumulative exposure but the chemical or substance is not being removed from the environment. This setup can be used to model morphogens (signaling molecules) that often act directly on a cell by binding to a receptor. It is well established that cells exposed to high levels of transforming growth factor (TGF)-$\beta$ will lead to cell death \cite{Fosslien09}. In this scenario, morphogens bind and initiate a secondary process in the cell that accumulates and leads to cell death. After a small period of time (smaller than the scale of movement or time to cell death), the morphogen is released from the receptor and thus the relative morphogen concentration can be assumed constant in time.}

\cso{Now that we have developed the initial framework, additional model extensions will allow us to look at real world examples for cells coupled to the environment in a cumulative way. 
There are many examples of cells that have directed motion due to a chemoattractant \cite{Swaney10}. 
This would involve extending the model beyond the general framework that studied a constant or fixed bias} \cmy{without affecting the chemical concentration} \cso{(Section \ref{sec: biased}).} 
\cmy{Cell movement will be biased due to the chemical gradient and, in turn, the chemical gradient will change due to localized cellular absorption.  We can derive the system of PDEs in a manner similar to that of Sections \ref{sec: Derivation} and \ref{sec: Absorption}.  The method of deriving the Heaviside term for $V$ in Section \ref{sec: Absorption} provides us the capability of handling different states as well as the capability of handling cells that die and no longer absorb chemicals.  However, further analytical work will be required since the well-posedness proof of the PDE and numerical solution will need to be modified to account for a new space and time dependent advection term as well as the fact that the absorption term, $\beta$, is coupled with the space and time dependent chemical profile.}

\cso{Further, there are other interesting biological examples of toxins that might not only kill the cell, but will change their motility. In wound healing, the migration of epithelial cells is important. However, exposure to chemical substances (e.g. chemical warfare agents) has been shown to also decrease cell migration speeds \cite{Steinritz15}. Another example is in cancer therapies where immunotoxins such as MOC31PE reduce cell migration and induce gene expression, leading to cell death in ovarian cancer cells \cite{Wiiger14}. 
Additionally, human neural crest cells migrate to particular locations in the embryo and differentiate into different cell types. It has been found that interferon-beta (IFN$\beta$), a potential toxin \cite{Pallocca17}, impairs neural crest cell migration. 
To model these types of applications, one would potentially have to re-derive the equations for a variable step size based on the cumulative exposure.  In terms of the computational method, we believe this case could be simulated} \cmy{using the same semi-discrete numerical formulation since the $\xi$-dependent diffusion parameter will be treated as a constant in the Green's function.}
\cmy{However, the $\xi$-dependent diffusion parameter will cause subtle complications when manipulating the integrals within the well-posedness analysis.}

Although we have focused on the example of a cell or agent absorbing a fraction of the particles of a chemical in the surrounding environment, these equations are generally applicable to any scenario where a cell is accumulating any quantity that is a function of space and/or time. In the current contexts, we have focused on the case of cells that do not interact with each other. The focus of future work will investigate interactions of different cells or agents as they interact in a cumulative way with their environment.

\section*{Acknowledgments}
This work was supported, in part, by National Science Foundation Division of Mathematical Sciences grant 1455270. S.D. Olson was also supported, in part, by a Fulbright research scholar grant. M.A. Yereniuk has been supported, in part, by the SMART scholarship program. All simulations were run on a cluster acquired through National Science Foundation grant 1337943.

\newpage
\bibliographystyle{plain}
\bibliography{ms}

\end{document}